\documentclass[11pt]{amsart}
\usepackage{subfigure}
\usepackage{epstopdf}
\usepackage{amssymb,amsmath}
\usepackage{caption}
\usepackage{float}
\usepackage{graphicx,subfigure,epsfig}
\usepackage{mathtools}
\DeclarePairedDelimiter\abs{\lvert}{\rvert}

\usepackage{amssymb,amsmath}
\usepackage{caption}
\usepackage{float}
\usepackage{multicol}
\usepackage{fancybox}
\usepackage[usenames, dvipsnames]{color}
\newtheorem{theorem}{Theorem}[section]
\newtheorem{definition}{Definition}[section]
\newtheorem{lemma}{Lemma}[section]
\newtheorem{remark}{Remark}[section]

\newtheorem{fact}{Fact}
\newtheorem{corollary}{Corollary}[section]
\newtheorem{example}{Example}[section]
\newtheorem{proposition}{Proposition}[section]

\usepackage{amssymb}
\begin{document}

\title[An Unknotting Index for Virtual Links]{An Unknotting Index for Virtual Links}

\author{Kirandeep KAUR}
\address{Department of Mathematics, Indian Institute of Technology Ropar, India}   
\email{kirandeep.kaur@iitrpr.ac.in} 

\author{Madeti PRABHAKAR}
\address{Department of Mathematics, Indian Institute of Technology Ropar, India}   
\email{prabhakar@iitrpr.ac.in}

\author{Andrei VESNIN}
\address{Tomsk State University, Tomsk, 634050, Russia \\ and Sobolev Institute of Mathematics, Novosibirsk, 630090, Russia} 
\email{vesnin@math.nsc.ru} 




\begin{abstract}
Given a virtual link diagram $D$, we define its unknotting index $U(D)$ to be minimum among $(m, n)$ tuples,  where $m$ stands for the number of crossings virtualized and $n$ stands for the number of classical crossing changes, to obtain a trivial link diagram. By using span of a diagram and linking number of a diagram we provide a lower bound for unknotting index of a virtual link. Then using warping degree of a diagram, we obtain an upper bound. Both these bounds are applied to find  unknotting index for virtual links obtained from pretzel links by virtualizing some crossings.
\end{abstract}





\maketitle

\section*{Introduction} \label{section1} 

Virtual knot theory was introduced by L.H.~Kauffman \cite{kauffman1999virtual} as a natural generalization of the theory of classical knots. Some knot invariants have been naturally extended to virtual knot invariants and, more generally, to virtual link invariants, as well. In the recent past, several invariants, like arrow polynomial~\cite{dye2009virtual}, index polynomial~\cite{im2010index}, multi-variable polynomial~\cite{manturov2003multi} and polynomial invariants of virtual knots~\cite{kpv-1} and links~\cite{silver2003polynomial} have been introduced to distinguish two given virtual knots or links. Another approach that can be extended from classical to virtual links to construct interesting invariant  is based on unknotting moves. One of the unknotting moves for virtual knots is known as \emph{virtualization}, which is a replacement of classical crossing by virtual crossing.  Observe, that classical unknotting move, that is replacement of a classical crossing to another type of classical crossing, is not an unknotting operation for virtual knots. 

In~\cite{unknotting}, K.~Kaur, S.~Kamada, A.~Kawauchi and M.~Prabhakar introduced an unknotting invariant for virtual knots, called an \emph{unknotting index for virtual knots}. We extend the concept of unknotting index for the case of virtual links and present lower and upper bound for this invariant. To demonstrate the method, bases on these bounds, we provide the unknotting index for a large class of virtual links obtained from pretzel links by applying virtualization moves to some crossings.  

This paper is organized as follows. Section~\ref{section1} contains preliminaries that are required to prove the main results of the paper. Namely, we define unknotting index for virtual links and review the concept of Gauss diagram for $n$-component virtual links. To obtain a lower bound on unknotting index, we define span of the virtual link and for an upper bound, we define warping degree for virtual links. In Section~\ref{section2}, we provide a lower bound for the unknotting index, see   Theorem~\ref{thm-bound}, and for upper bound, see  Theorem~\ref{thm-upb}. Using these bounds, in Section~\ref{section3} we determine  unknotting index for large class of virtual links that are obtained from classical pretzel links  by virtualizing some classical crossings.

\section{Preliminaries} \label{section1} 

A diagram of virtual link has two type of crossings: (classical) crossings and virtual crossings. In pictures given below virtual crossings are encircled by a small circles. Two virtual link diagrams are said to be \emph{equivalent} if one can be deformed into another by using a finite sequence of classical Reidemeister moves RI, RII, RIII, as shown in Fig.~\ref{fig1a}, and virtual Reidemeister moves VRI, VRII, VRIII, SV, as shown in Fig.~\ref{fig1b}.
\begin{figure}[!ht]
\smallskip 
\begin{center}
\subfigure[Classical Reidemeister moves.]{\includegraphics[scale=0.45]{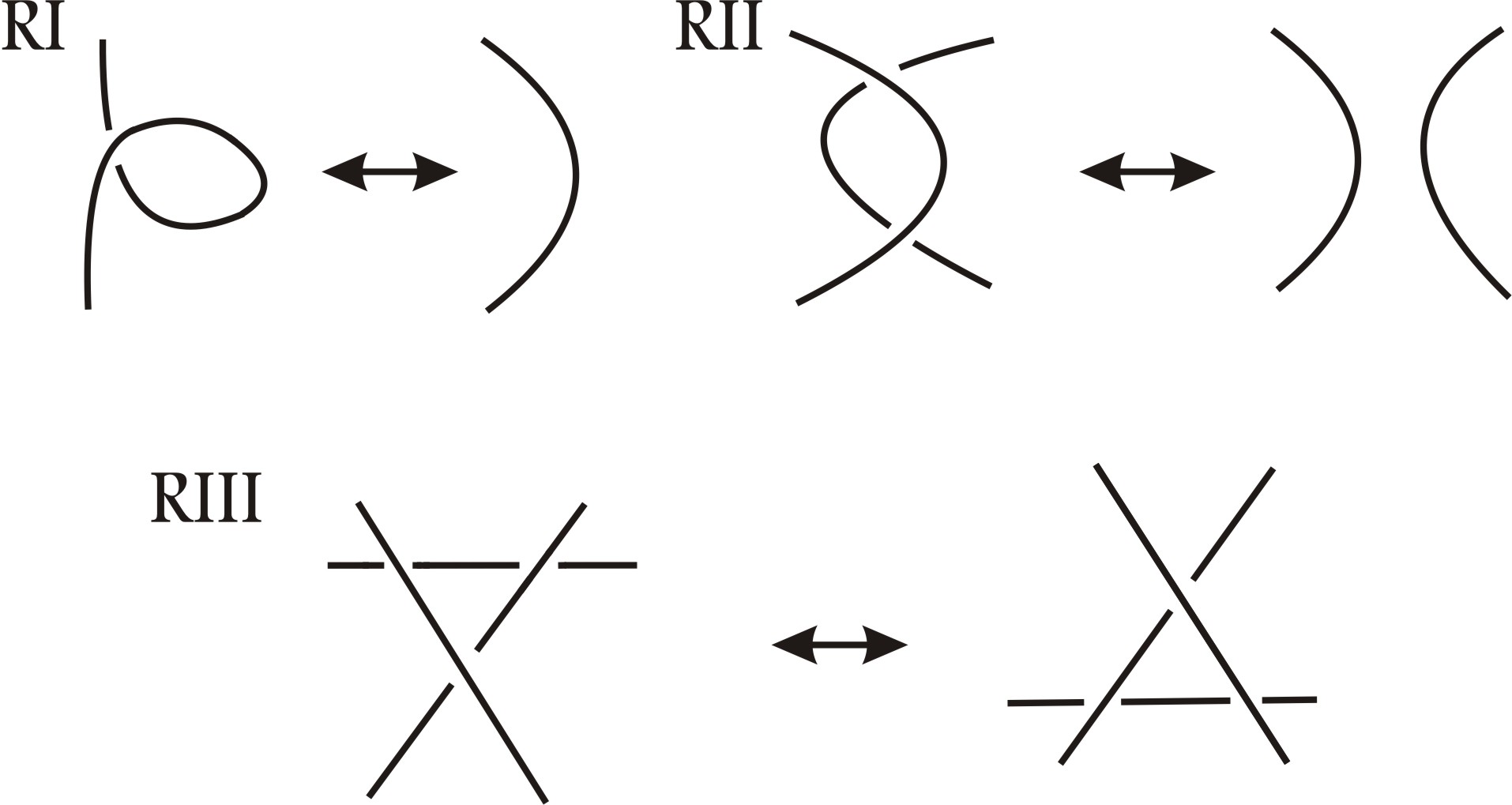} \label{fig1a}}
\subfigure[Virtual Reidemeister moves.]{\includegraphics[scale=0.45]{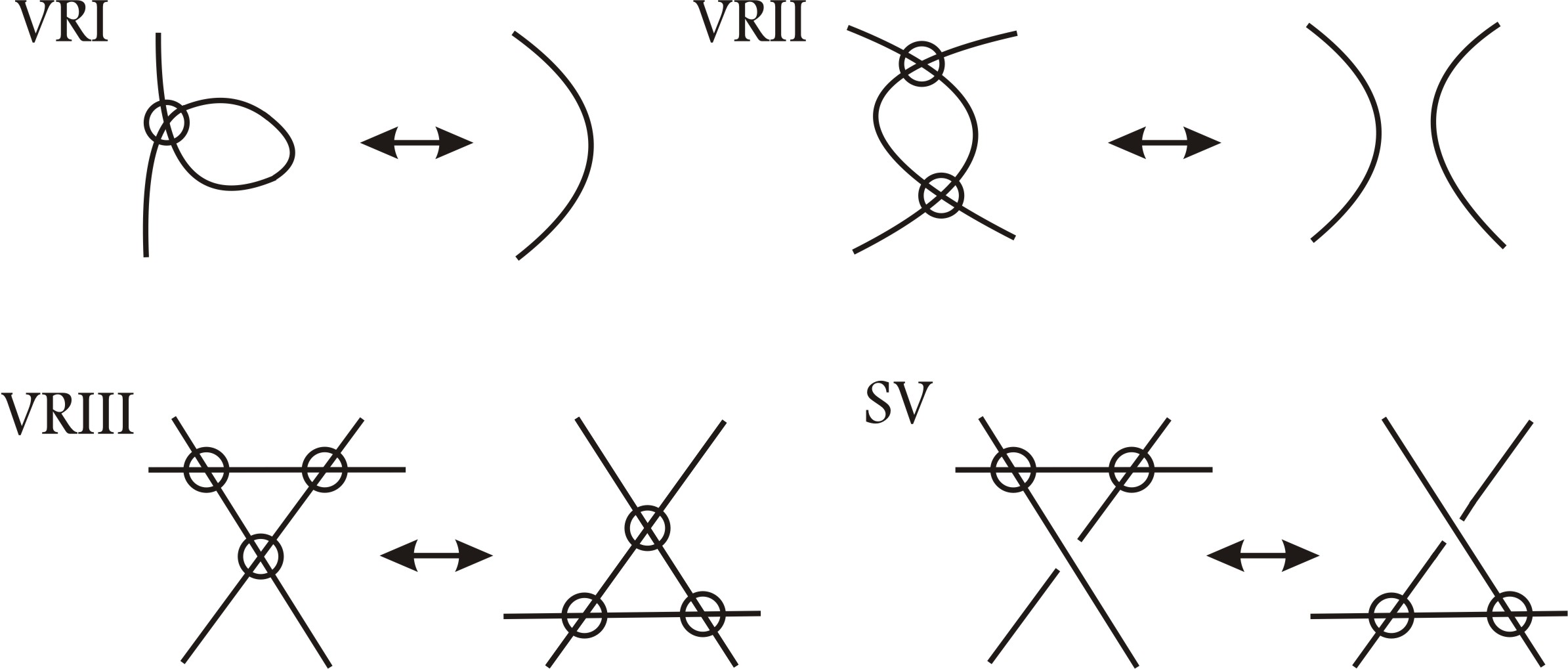} \label{fig1b}} 
\caption{Two kinds of Reidemeister moves.} \label{rm}
\end{center} 
\end{figure}

Given a virtual link diagram $D$ and an ordered pair $(m,n)$ of non negative integers, the diagram $D$ is said to be $(m,n)$-\emph{unknottable} if, by virtualizing $m$ classical crossings and by applying crossing change operation to $n$ classical crossings of $D$, the resulting diagram can deformed into a diagram of a trivial link.  Obviously, if $D$ has $c(D)$ crossings, then $D$ is $(c(D),0)$-unknottable. We define  \emph{unknotting index} of $D$, denoted by $U(D)$, to be minimum among all such pairs $(m,n)$ for which $D$ is $(m,n)$-unknottable. Here the minimality is taken with respect to the dictionary ordering. In Fig.~\ref{fig2}, we present examples of virtual link diagrams and their unknotting index, which are easy to compute.
 \begin{figure}[!ht]
{\centering
\subfigure[$U(D)= (1,0)$]{\includegraphics[scale=0.36]{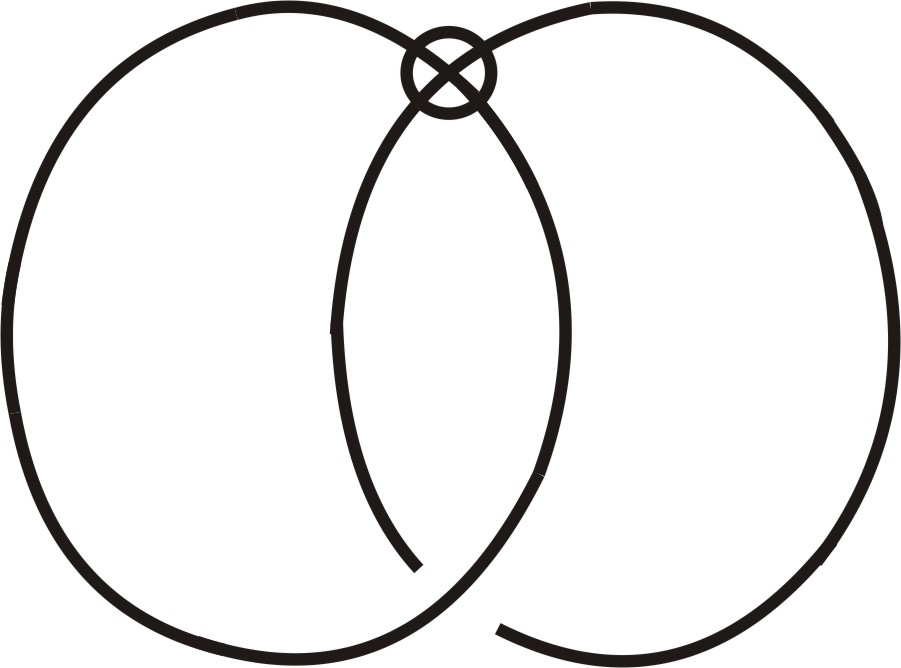} \label{hopf}}
 \hspace{1.cm}
\subfigure[$U(D)=(1,0)$]{\includegraphics[scale=0.4]{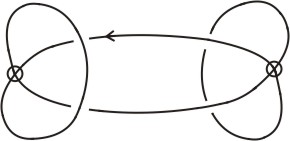} \label{knkt} }
 \hspace{1.cm}
\subfigure[$U(D)=(0,2)$] {\includegraphics[scale=0.38]{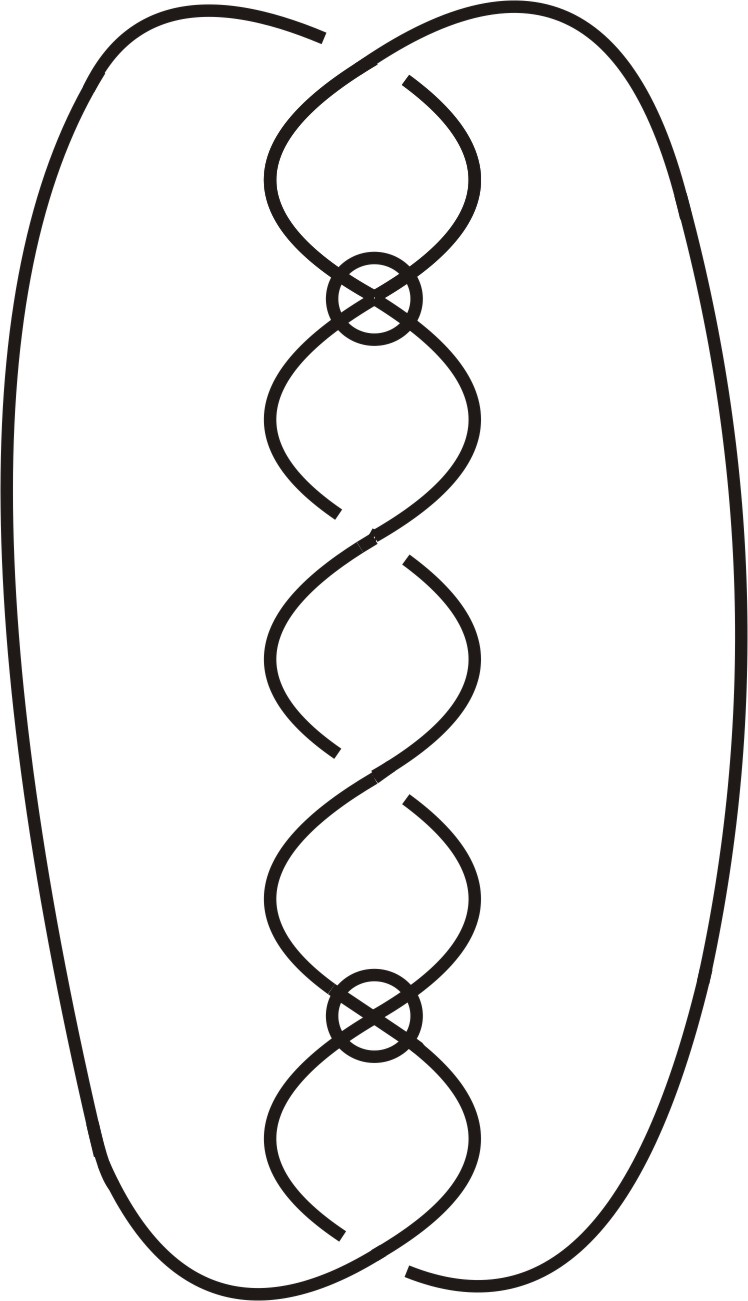} }
\caption{Virtual link diagrams and their unknotting indices.} \label{fig2}
}
\end{figure} 

\begin{definition} 
{\rm The \emph{unknotting index} $U(L)$ of a virtual link $L$ is defined as $U(L)=\min U(D)$, where minimum is taken over all diagrams $D$ of $L$.
}
\end{definition}
 
It is easy to observe that a virtual link $L$ is trivial if and only if $ U(L)=(0,0)$. For classical link $L$, it is obvious to see that $U(L)\leq (0,u(L))$, where $u(L)$ is the usual unknotting number of $L$.

In general, it is a difficult problem to find the unknotting index for a given virtual link. In case of virtual knots, some lower bounds are provided on this unknotting index in~\cite{unknotting} using $n$-$th$ writhe invariant  $J_n (K)$, introduced in~\cite{satoh2014writhes}, see definition~\ref{def1.6}. 

\begin{proposition} {\rm \cite[Proposition~4.2]{unknotting},~\cite[Theorem~1.5]{satoh2014writhes} } \label{p1}
Let $K$ be a virtual knot. Then the following properties hold:  
\begin{itemize}
\item[(1)]  If $J_k(K) \neq J_{-k}(K)$ for some $k \in {\mathbb Z}\setminus \{0\}$, then $(1,0) \leq U(K)$.
\item[(2)]  $(0, \frac{1}{2} \sum_{k \neq 0} | J_k(K) | ) \leq U(K)$.
\end{itemize}
\end{proposition}

A \emph{flat virtual knot diagram} is a virtual knot diagram with ignoring over/un\-der information at crossings. A virtual knot diagram $D$ can be deformed into unknot by applying crossing change operations if and only if the flat virtual knot diagram corresponding to $D$ presents the trivial flat virtual knot. By Proposition~\ref{p1}, the flat virtual knot,  corresponding to $K$, is non-trivial if there exists an integer $k$ such that $J_k(K) \neq J_{-k}(K)$.

Now, let us turn to the case of virtual links. We will provide a lower bound on the unknotting index for a given virtual link. Namely, we will modify the lower bound given in Proposition~\ref{p1} using $\operatorname{span}$ and linking number of  diagram. We recall the definition of linking number and Gauss diagram and review span invariant, which we use to find the lower bound.   

\begin{definition} {\rm 
For an $n$-component virtual link $L=K_{1} \cup K_{2}\cup \ldots \cup K_{n}$, the \emph{linking number} $\operatorname{lk} (L)$ is defined as
$$
\operatorname{lk} (L)=\frac{1}{2}\sum_{c_{k}\in K_{i} \cap K_{j}, \text{\rm with} i \neq j} \operatorname{sgn}(c_{k}),
$$
where $\operatorname{sgn}(c_{k})$ is the sign of $c_{k}$, defined as in Fig.~\ref{fig3}. 
\begin{figure}[!ht]
\centering 
\unitlength=0.6mm
\begin{picture}(0,35)(0,-5)
\thicklines
\qbezier(-40,10)(-40,10)(-20,30)
\qbezier(-40,30)(-40,30)(-32,22) 
\qbezier(-20,10)(-20,10)(-28,18)
\put(-35,25){\vector(-1,1){5}}
\put(-25,25){\vector(1,1){5}}
\put(-30,0){\makebox(0,0)[cc]{$\operatorname{sgn}(c)=+1$}}
\qbezier(40,10)(40,10)(20,30)
\qbezier(40,30)(40,30)(32,22) 
\qbezier(20,10)(20,10)(28,18)
\put(25,25){\vector(-1,1){5}}
\put(35,25){\vector(1,1){5}}
\put(30,0){\makebox(0,0)[cc]{$\operatorname{sgn}(c)=-1$}}
\end{picture}
\caption{Sign of crossing $c$.} \label{fig3}
\end{figure}
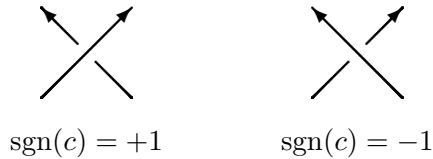 
}
\end{definition}   

In \cite{cheng2013polynomial}, Z.~Cheng and H.~Gao defined an invariant, called  \emph{span}, for 2-component virtual links using Gauss diagram. Remark that span is same as the absolute value of wriggle number provided by L. C.~Folwaczny and L. H.~Kauffman in \cite{folwaczny2013linking}. Consider a diagram $D=D_{1}\cup D_{2}$ of a virtual link $L = K_1 \cup K_2$. Let us traverse along $D_{1}$ and consider crossings of $D_{1}$ and $D_{2}$. If $r_{+}$ (respectively, $r_{-})$ is the number of over linking crossings with positive sign (respectively, negative sign) and $\ell_{+}$ (respectively, $\ell_{-})$ is the number of under linking crossings with positive sign (respectively, negative sign), then $\operatorname{span}(D)$ of $D$ is defined as
$$
\operatorname{span}(D)= | r_{+} - r_{-} - \ell_{+} + \ell_{-} |.
$$
It is easy to see, that we will get the same result by traverse along $D_2$. Since, due to~\cite{cheng2013polynomial}, $\operatorname{span} (D)$ of diagram $D$ of a link $L$ is an invariant for $L$, we denote it by $\operatorname{span}(L)$.  It is easy to see, that for a classical 2-component link $L$ we get $\operatorname{span}(L) = 0$. 

\begin{definition}{\rm 
For a virtual link $L=K_{1}\cup K_{2} \cup \ldots \cup K_{n}$, we define \emph{span} of $L$ as 
$$ 
\operatorname{span} (L) =  \sum _{i \neq j} \operatorname{span} (K_{i} \cup K_{j}).
$$
}
\end{definition}

\noindent
Since $\operatorname{span} (K_{i} \cup K_{j})$ is a virtual link invariant, $\operatorname{span} (L)$ is also a virtual link invariant. 

The following property is obvious and we state it as Lemma for further references. 

\begin{lemma} \label{lemma1.1}
If  $D'$ is a diagram obtained from a 2-component virtual link diagram $D$ by virtualizing one crossing, then $|\operatorname{span} (D) - \operatorname{span} (D') | \leq 1$.  
\end{lemma} 

It is obvious that $\operatorname{span}(L)$ leaves invariant under crossing change operation. Also,  for two equivalent 2-component virtual link diagrams, $D=D_{1}\cup D_{2}$ and $D'=D'_{1}\cup D'_{2}$, their linking crossings are related as $r'_{+} = r_{+}+ s$, $r'_{-}=r_{-} + s$ for some $s \in \mathbb Z$ and $\ell'_{+} = \ell_{+} + t$, $\ell'_{-} = \ell_{-} + t$, for some $t \in \mathbb{Z}$.

The $\operatorname{span}(L)$ of a virtual link $L$ can be calculated through Gauss diagrams. We define Gauss diagram for an oriented $n$-component virtual link as follows. 

\begin{definition} {\rm \emph{Gauss diagram} $G(D)$ of an $n$-component virtual link diagram $D$ consists of $n$ oriented circles with over/under passing information in crossings be presented by directed  chords and segments. For a given crossing $c \in D$ the chord (or segment) in $G(D)$  is directed from over crossing $\overline{c}$ to under crossing $\underline{c}$.   
}
\end{definition}
 
Fig.~\ref{fig4b} depicts the Gauss diagram corresponding to the virtual link diagram presented in~Fig.~\ref{fig4a}.
\begin{figure}[!ht]
{\centering
\subfigure[$D= D _{1}\cup D_{2}$]{ \includegraphics[scale=0.7]{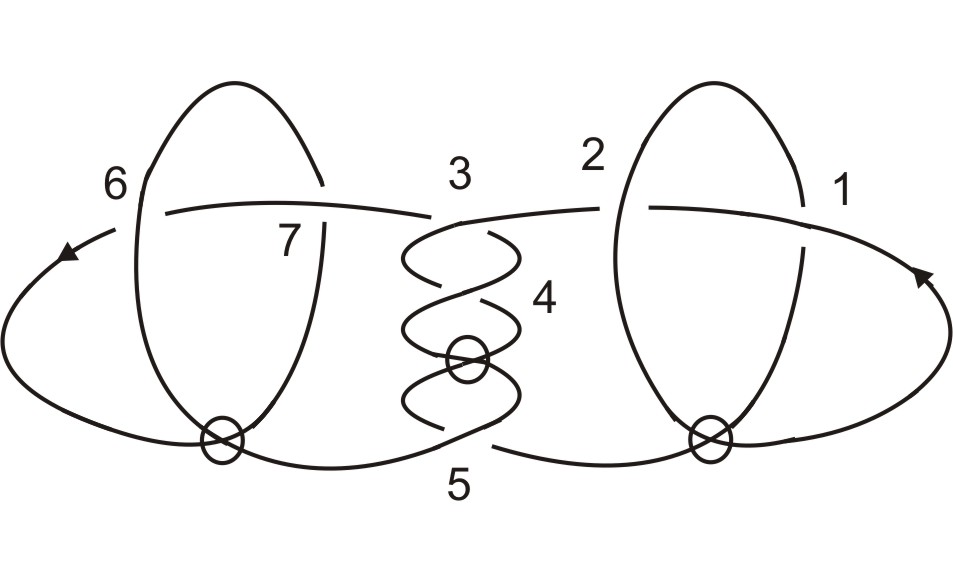} \label{fig4a}}
\hspace{1.cm}
\subfigure[$G(D)=G( D _{1} \cup D_{2})$] {\includegraphics[scale=0.6]{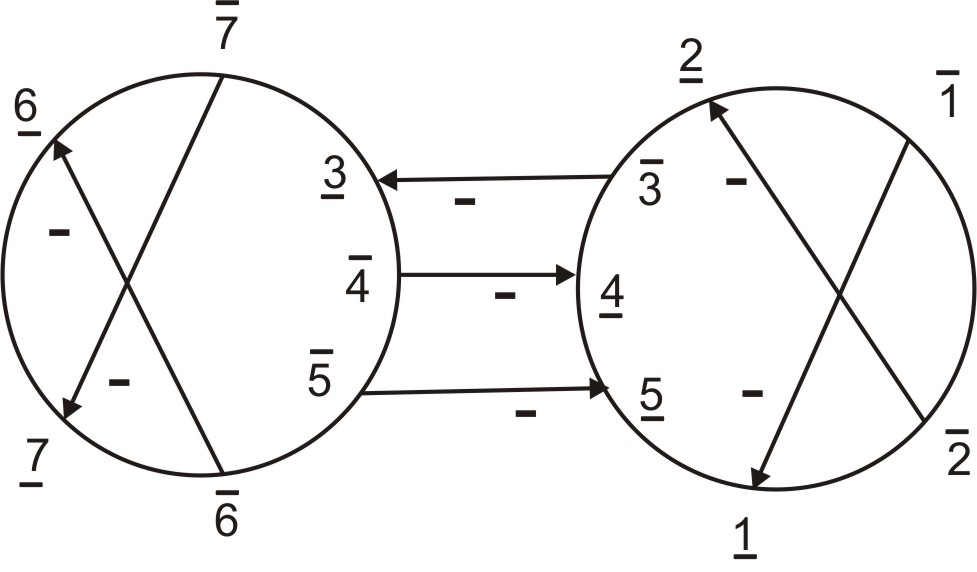} \label{fig4b}}
\caption{Virtual link diagram $D$ and its Gauss diagram $G(D)$.}
}
\end{figure}

\noindent In \cite{cheng2013polynomial}, Z. Cheng and H. Gao assigned an integer value, called \emph{index value}, to each classical crossing $c$ of a virtual knot diagram using Gauss diagrams and denoted it by $\operatorname{Ind}(c)$. 
\begin{definition}{\rm
Let $D$ be a virtual knot diagram and $\gamma_c$ be a chord of Gauss diagram $G(D)$.  
Let $\mathrm{r}_{+}$  (respectively, $\mathrm{r}_{-}$) be the number of positive  (respectively, negative) chords intersecting  $\gamma_c$ transversely from right to left as shown in Fig~\ref{Fig5}.
Let $l_{+}$  (respectively, $l_{-}$) be the number of positive (respectively, negative) chords intersecting $\gamma_{c}$ transversely from left to right.   
Then the {\it  index}  of $\gamma_c$ is defined as 
\[{\rm Ind}(\gamma_c)=\mathrm{r}_{+}-\mathrm{r}_{-}-l_{+}+ l_{-}.\]}
\end{definition}
\begin{figure}[!ht] 
{
\centering
\includegraphics[scale=0.38]{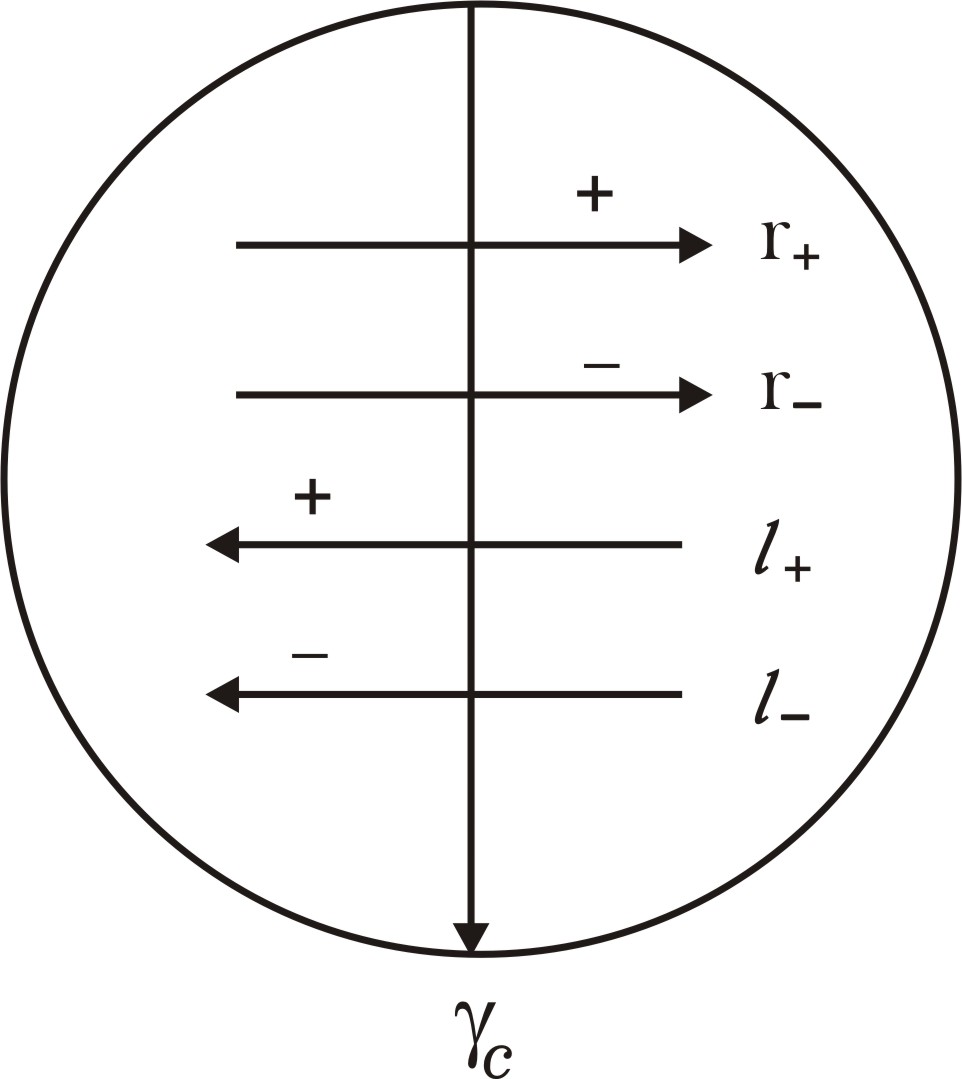} 
\caption{Chords intersecting transversely to a chord $\gamma_c$}
\label{Fig5}
 }
\end{figure}
\noindent The {\it index value} $\operatorname{Ind}(c)$ of a crossing $c$ in $D$ is given by the {\it index value} ${\rm Ind} (\gamma_{c})$ of the corresponding chord $\gamma_{c}$ in $G(D)$.

\begin{definition} \label{def1.6}
{\rm For each $n \in \mathbb{Z}\setminus \{0\}$, the  \emph{$n$-th writhe $J_n(D)$} of an oriented virtual knot diagram $D$ is defined as the sum of signs of those crossings in $D$, whose index value is $n.$ Hence,
\[J_n(D)=\displaystyle \sum_{c:~{\rm Ind}(c)=n}\operatorname{sgn}(c).\] 
By \cite{satoh2014writhes} the $n$-th writhe, $J_{n}(D)$, is a virtual knot invariant.}
\end{definition}

To obtain an upper bound on the unknotting index, we define warping degree for virtual links. In \cite{shimizu2011warping}, A.~Shimizu defined warping crossing points for a link diagram. Here we use the same terminology for virtual links. Let $D=D^{1}\cup D^{2} \cup \ldots \cup D^{n}$ be an orientated virtual link diagram and $D_{a}=D^{1}_{a_{1}} \cup D^{2}_{a_{2}} \cup \ldots \cup D^{n}_{a_{n}}$ denotes the based diagram of $D$ with the base point sequence $a= \langle a_{1},a_{2}, \ldots, a_{n}\rangle$, where $a_i$ is a non-crossing point on $D^{i}$ for each $1\leq i\leq n$. A self crossing $c$ in $D^{i}_{a_i}$ is said to be a warping crossing point, if we encounter $c$ first at under crossing point while moving from $a_i$ along the orientation in $D^{i}_{a_i}$. A linking crossing $c$ between $D^{i}_{a_i}$ and $D^{j}_{a_j}$ is said to be a warping crossing point, if $c$ is an under crossing of $D^{i}_{a_i}$ for $1\leq i<j \leq n$.

Then the \emph{warping degree} of $D_{a}$, denoted by $d(D_{a})$, is defined as the minimum number of crossing points that have to change in $D_{a}$ from under to over starting from $a_{i}$ in each $D^{i}_{a_{i}}$, such that the resulting based diagram with base point sequence $\langle a_{1},a_{2}, \ldots, a_{n}\rangle$ has no warping crossing point.
 
\begin{definition}{\rm 
The \emph{warping degree} of a virtual link diagram $D$ is defined as
$$
d(D)= \min \lbrace d(D_{a}) \mid a ~\mbox{is a base point sequence} \rbrace.
$$
}
\end{definition}

If $D$ is a classical link diagram with $d(D)=0$, then $D$ presents a trivial link. This is in general not true in case of virtual link diagrams. The warping degree is zero for the virtual link diagrams shown in Fig.~\ref{hopf} and Fig.~\ref{5}, even though these diagrams does not present trivial link. Moreover, if $D$ is diagram of classical link, then $u(D)\leq d(D)$. But this is in general not true for virtual links whose usual unknotting number exist. For virtual trefoil knot diagram shown in Fig.~\ref{virtual_trefoil}, we have $u(D)=1$ and $d(D)=0$, thus $u(D)\nleq d(D)$. In Section~\ref{section2}, we will use warping degree to establish an upper bound on unknotting number for virtual links.
\begin{figure}[!ht] 
{\centering
\includegraphics[scale=0.4]{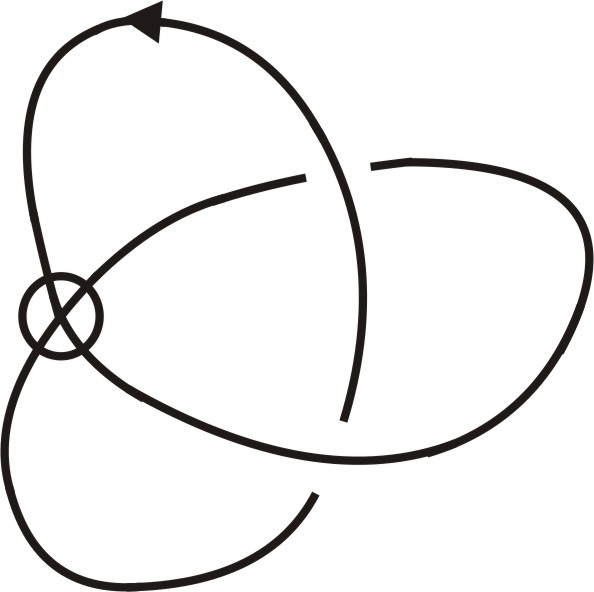} 
\caption{A diagram of virtual trefoil knot.}\label{virtual_trefoil}
 }
\end{figure}

\section{Bounds on Unknotting Index}  \label{section2}

In this section, we will provide bounds on  unknotting index for virtual links.

\begin{lemma} \label{lem:span}
Let $D$ be a virtual link diagram. Then $\operatorname{span}(D)$ is equal to the minimum number of crossings in $D$ which should be virtualized to obtain a diagram $D'$ such that $\operatorname{span}(D')=0$.  
\end{lemma}

\begin{proof} 
Let $D=D^{0} = D^{0}_{1} \cup \ldots \cup D^{0}_{n}$ be an $n$-component virtual link diagram, and $Âý D^{m} = D^{m}_{1} \cup \ldots \cup D^{m}_{n}$ be a diagram with $\operatorname{span}(D^{m}) = 0$. Let $\{ D^{0}, D^{1}, \ldots, D^{m} \}$ be a sequence of $n$-components virtual link diagrams, where $D^{i+1} = D^{i+1}_{1} \cup \ldots \cup D^{i+1}_{n}$ is obtained from $D^{i} = D^{i}_{1} \cup \ldots \cup D^{i}_{n}$ by virtualizing exactly one crossing. Denote this crossing by $c_{i}$ and suppose that $c_{i} \in D^{i}_{k(i)} \cup D^{i}_{\ell(i)}$ for some $k(i)$ and $\ell(i)$. Then
$$ 
\begin{gathered}
\operatorname{span} (D) = | \operatorname{span} (D^{0}) - \operatorname{span}(D^{m}) | = \left| \sum _{i=0}^{m-1} \left( \operatorname{span}(D^{i}) - \operatorname{span} (D^{i+1}) \right) \right|  \\ \leq   \sum _{i=0}^{m-1} \left| \operatorname{span}(D^{i}) - \operatorname{span} (D^{i+1}) \right| \\ = \sum_{i=0}^{m-1} \left| \operatorname{span} (D^{i}_{k(i)} \cup D^{i}_{\ell(i)}) - \operatorname{span} (D^{i+1}_{k(i)} \cup D^{i+1}_{\ell(i)} ) \right|= m, 
\end{gathered} 
$$
where for the last step we used Lemma~\ref{lemma1.1}. 

To obtain the inverse inequality, let us start with virtual link diagram $D=D^{0}_{1} \cup \ldots \cup D^{0}_{n}$ and consider a pair of components $D^{0}_{i}\cup D^{0}_{j}, ~i < j$, and traverse along $D^{0}_{i}$. Let $r_{ij+}$ (respectively, $r_{ij-})$ be the number of over linking crossings with positive sign (respectively, negative sign) and $\ell_{ij+}$ (respectively, $\ell_{ij-})$ be the number of under linking crossings with positive sign (respectively, negative sign). Then $\operatorname{span} (D^{0}_{i} \cup D^{0}_{j}) = |r_{ij+} - r_{ij-} - \ell_{ij+} + \ell_{ij-}|$. 
  Now, in each $D^{0}_{i}\cup D^{0}_{j}, ~i < j$, virtualize $\operatorname{span}(D^{0}_{i}\cup D^{0}_{j} )$ number of crossings as follows; 
\begin{itemize}
\item If $ r_{ij+}+\ell_{ij-}> r_{ij-}+\ell_{ij+}$,  then we virtualize a crossing which is either a positive sign over crossing or a negative sign under crossing, and
\item if $r_{ij+} + \ell_{ij-} < r_{ij-} + \ell_{ij+}$, then we virtualize a crossing which is either a negative sign over crossing or a positive sign under crossing. 
 \end{itemize}
 
\noindent After virtualizing  $\displaystyle \sum _{i\neq j} \operatorname{span} (D^{0}_{i} \cup D^{0}_{j}) $ number of crossings in $D$,  the resulting diagram has zero span value and $ m\leq \operatorname{span(D)} $. Hence $m=\operatorname{span}(D)$. 
\end{proof}

\begin{corollary} \label{c1}
If $L$ is an $n$-component virtual link, then  $U(L)\geq (\operatorname{span}(L),0)$.
\end{corollary}

\begin{proof} Observe that $\operatorname{span} (L)$ of virtual link $L$ is invariant under crossing change operation. Thus, if $\operatorname{span} (L) \neq 0$ then $L$ is non-classical link. Hence by Lemma~\ref{lem:span}, $U(L)\geq (\operatorname{span}(L),0 )$. 
 \end{proof}

\begin{remark}{\rm 
If $L$ is a virtual link with $U(L)=(m,n)$ and $\operatorname{span} (L) =0$, then $m$ need not be zero. For example, span of the virtual link presented by the diagram given in Fig.~\ref{5} is zero, but the  unknotting index is $(2,0)$.}
\end{remark}

\begin{figure}[!ht]
{
\centering
\subfigure[Virtual  link diagram $D$]
{
\includegraphics[scale=0.4]{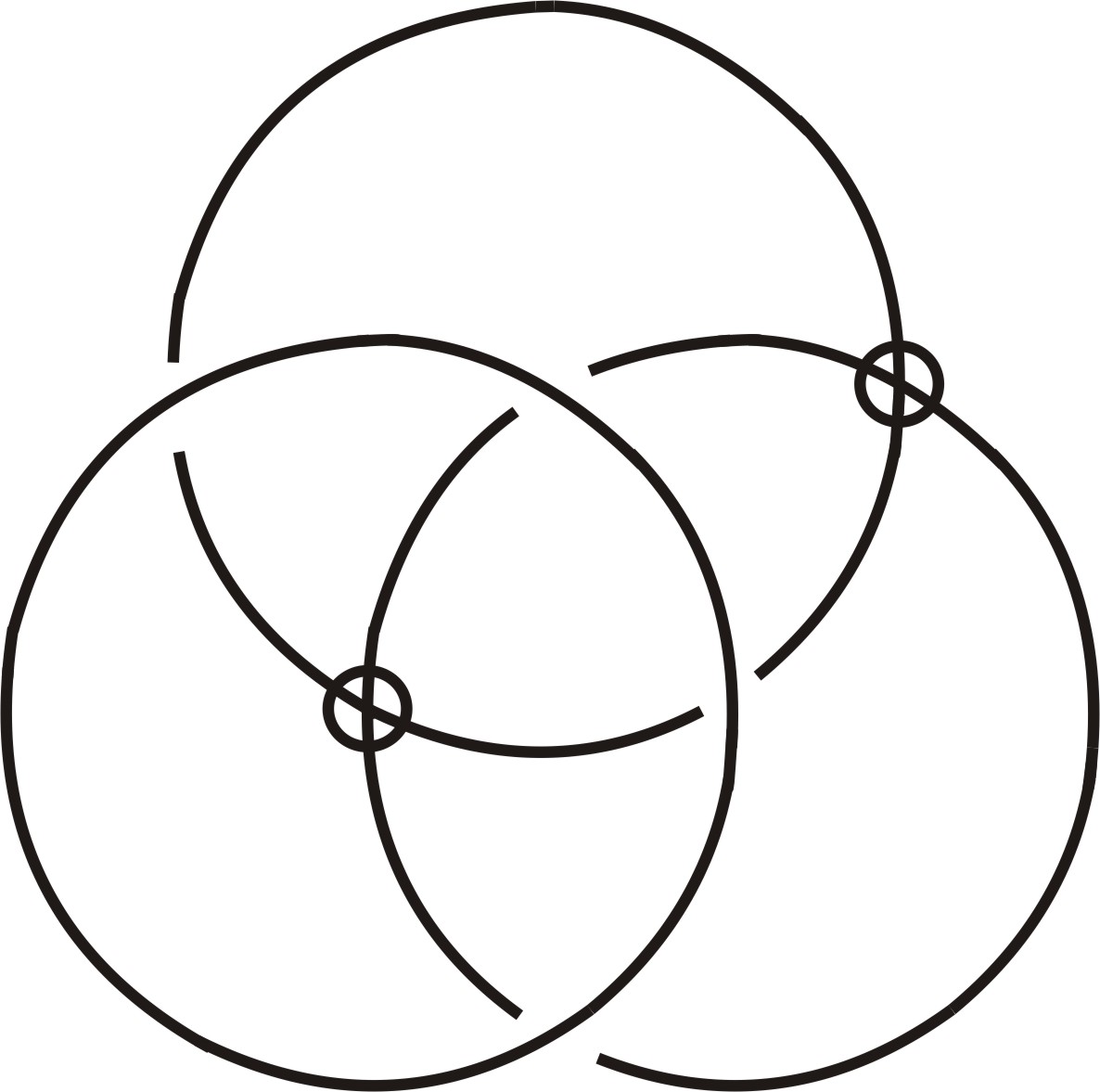}
\label{5}
}
 \hspace{.8cm}
  \subfigure[$P_{L}(t)=2t^{-1}-t^{-2}-1$]
{
\includegraphics[scale=0.4]{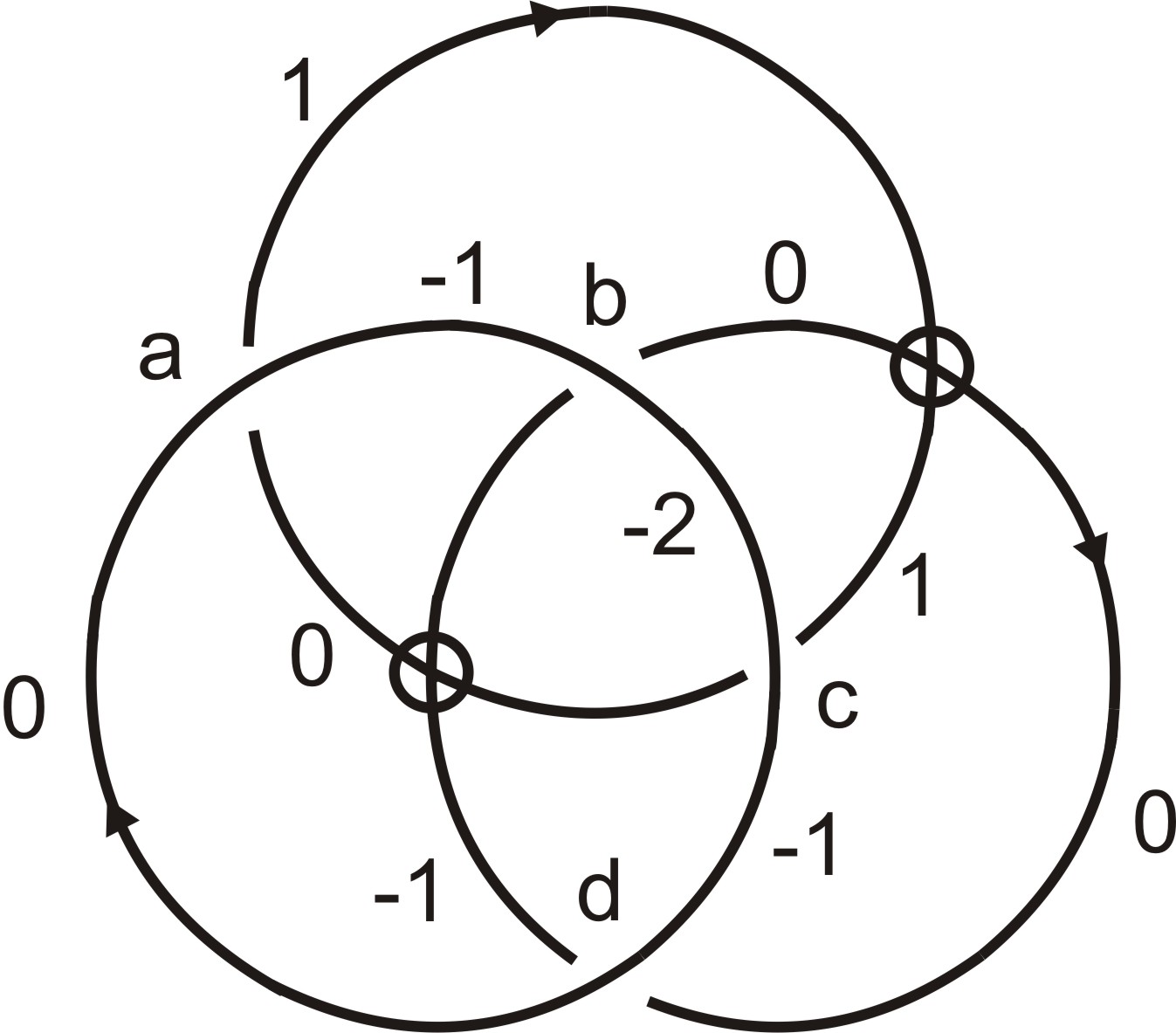}
\label{6}
 }
\caption{A virtual link and its labelling for an affine index polynomial.}
 }
 \end{figure}
 
\begin{example} {\rm 
Let $L$ be a virtual link represented by the diagram $D$ as shown in Fig.~\ref{5}.
Observe that the affine index polynomial, $P_L(t)$, shown in Fig.~\ref{6} never reduces to zero by changing crossings in $L$. Therefore the flat virtual link corresponding to $L$ is non trivial and at least one virtualization is needed to turn $L$ to unlink.  After one virtualization in $D$ the resulting diagram, say $L'$, has $\operatorname{span}(L')=1$. Thus $(2,0)\leq U(L)$ and by virtualizing crossing $a$ and $c$, $L$ can be deformed to trivial link.
}
\end{example}
 
 Suppose that $L$ is a virtual link with $U(L)=(m,n)$. From Corollary~\ref{c1}, $(\operatorname{span}(L),0)$  is a lower bound on $(m,n)$. By fixing $\operatorname{span}(L)$ as a lower bound on $m$, we establish a lower bound on $n$ using the concept of linking number. For this we introduce some notions as follows. 

Let $C(D)$ be the set of all classical crossings in a virtual link diagram $D$ and $S$ be any subset of $C(D)$. Denote $D_{S}$ as a diagram obtained from $D$ by virtualizing all the crossings of $S$. Denote the cardinality of the set $S$ by $|S|$. Now consider $\Lambda(D) $ defined by 
$$
\Lambda(D) =\{ S \subseteq C(D) \mid  |S| = \operatorname{span}(D) \quad \text{and} \quad  \operatorname{span} (D_{S})=0\}.
$$

Let us consider $\ell_D= \min \{ \abs{ \operatorname{lk}(D_{S})} \mid  S \in \Lambda(D) \}$. 

\begin{example}
The value $\ell_{D}$ is not an invariant for $L$. It is shown in Fig.~\ref{ld}, that for two equivalent diagrams $D$ and $D'$ of a 3-component link $L$ values $\ell_{D}$ and $\ell_{D'}$ are not equal. 
\begin{figure}[!ht]
{\centering
\subfigure[$D$]{\includegraphics[scale=0.5]{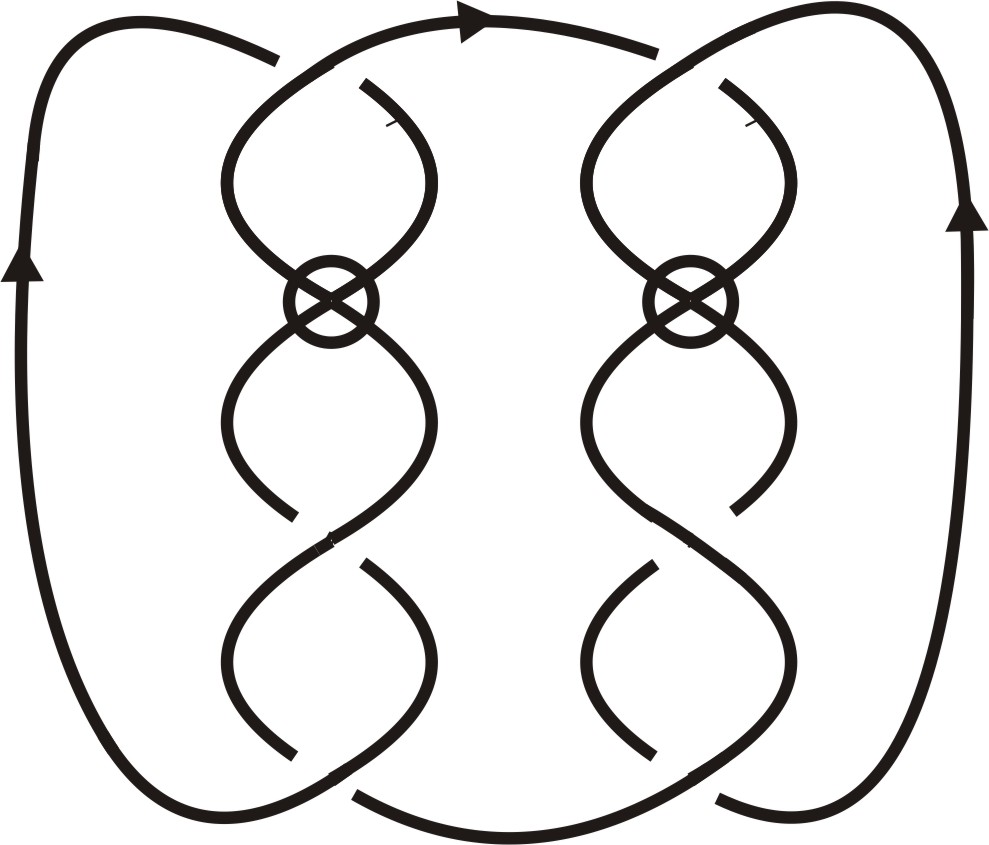} \label{tw}}
\hspace{.8cm}
\subfigure[$D'$]{\includegraphics[scale=0.5]{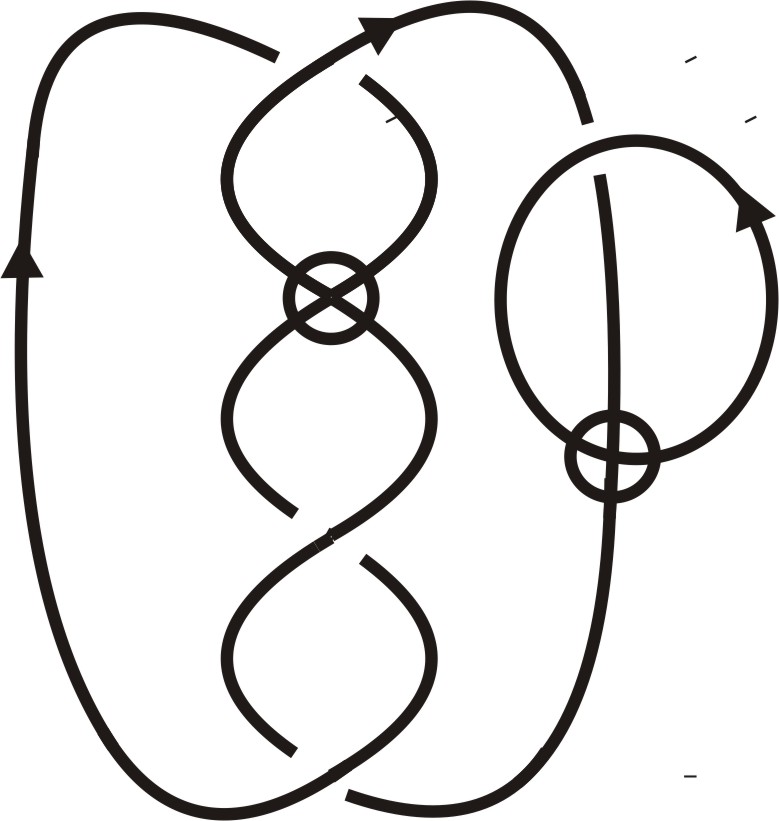} \label{x3}}
\caption{Diagrams $D$ and $D'$ with $\ell_{D}=0$ and $\ell_{D'}=1$.}
\label{ld}
}
\end{figure}
\end{example}  

But in case of $2$-component virtual links,  we observe that $\ell_{D}$ is a virtual link invariant.

\begin{theorem} \label{thm-span}
If $D=D_{1}\cup D_{2}$ is a virtual link diagram, then either $\ell_D= 0$ or  
$$
\ell_D=\abs{\operatorname{lk}(D)}-\frac{1}{2}\operatorname{span} (D) .
$$
\end{theorem}

\begin{proof} Let $D=D_{1}\cup D_{2}$ be a 2-component virtual link diagram. Without loss of generality consider traversing along $D_{1}$. Recall that 
$$\operatorname{span} (D) = \abs{ r_{+} + \ell_{-} - r_{-} - \ell_{+}},$$
where $r_{+}$ (respectively, $r_{-})$ is the number of positive (respectively, negative) over linking crossings and $\ell_{+}$ (respectively, $\ell_{-})$ is the number of positive (respectively, negative) under linking crossings.

We will show that either $\ell_D= 0$ or $ \ell_D=\abs{\operatorname{lk}(D)}-\frac{1}{2}\operatorname{span} (D).$  

If  $\operatorname{span}(D)=0$,  then $\Lambda(D)$ contains only empty set and hence $\ell_D = \abs{\operatorname{lk}(D)}$. 

Suppose $\operatorname{span}{D}\neq 0.$ Then $\Lambda(D)\neq \emptyset$. Let $S\in \Lambda(D)$, and $n$ and $m$ be number of positive and negative crossings, respectively, of $D$ that belong to $S$. Then $n+m=\operatorname{span}(D)$ and $\operatorname{lk}(D_S)=\operatorname{lk}(D)-(n-m)/2.$ 

Since $\operatorname{span}(D) \neq 0$,  we have either $r_{+} + \ell_{-} > r_{-} + \ell_{+}$ or $r_{+} + \ell_{-} < r_{-} + \ell_{+}$.  It is easy to observe that when $r_{+} + \ell_{-} > r_{-} + \ell_{+}$, the $n$ positive  crossings are over linking crossings and $m$ negative crossings are under linking crossings, whereas in  the case $r_{+} + \ell_{-} < r_{-} + \ell_{+}$, the $n$ positive crossings are under linking crossings and $m$ negative crossings are over linking crossings.

\smallskip 

\underline{\it Case 1.} 
Assume $r_{+} + \ell_{-} > r_{-} + \ell_{+}$. Then we have the following two subcases.  

\smallskip 

\underline{\it Subcase 1.1.} Assume $\ell_+> \ell_-$. Then $2\operatorname{lk}(D)> \operatorname{span}(D)$ and $r_{+}> \operatorname{span}(D)$. 
 
 Since $r_{+}> \operatorname{span}(D)$, we can ensure that there exist an $S\in \Lambda(D)$ such that the number of positive over linking crossings of $D$ in $S$ is $\operatorname{span}(D)$, i.e.,  $n=\operatorname{span}(D)$ and $m=0$. Therefore, 
\[\operatorname{lk}(D_S)= \operatorname{lk}(D)-\dfrac{1}{2}(n-m)=\operatorname{lk}(D)-\dfrac{1}{2}\operatorname{span}(D)>0.\] 

Hence  $\abs{\operatorname{lk}(D_S)}=\abs{\operatorname{lk}(D)-\dfrac{1}{2}\operatorname{span}(D)}$, which can be written as 
\begin{equation}\label{EQ1} \abs{\operatorname{lk}(D_S)} = \operatorname{lk}(D) - \dfrac{1}{2}\operatorname{span}(D), 
\end{equation}
since  $\operatorname{lk}(D)>0$. 
To be specific, for every set $S\in \Lambda(D)$ we have $\abs{\operatorname{lk}(D_S)}>0$. Indeed, if there exists a set $S'\in \Lambda(D)$ such that $\abs{\operatorname{lk}(D_{S'})}=0$, then $n'+m'=\operatorname{span}(D)$, where   $n'$ and $m'$ are the number of positive and negative crossings of $D$ that belongs to $S'$, and
$$
\operatorname{lk}(D_{S'})=\operatorname{lk}(D)-\dfrac{1}{2}(n'-m')=0 
$$
implies   
$$
2\operatorname{lk}(D)=n'-m'\leq n'+m',
$$ 
whence 
$$
2\operatorname{lk}(D)\leq \operatorname{span}(D),
$$ 
that gives a contradiction. Therefore for every set $S\in \Lambda(D)$ we have $\abs{\operatorname{lk}(D_S)}>0$ and hence $\ell_D \neq 0$.  

Now we need to show that $\abs{\operatorname{lk}(D)}-\operatorname{span}(D)/2 \leq \ell_D$. For this let us assume that there exist a set $S''\in \Lambda(D)$ such that  $\abs{\operatorname{lk}(D_{S''})}<\abs{\operatorname{lk}(D)}-\operatorname{span}(D)/2$. Then $n''+m''=\operatorname{span}(D)$, where $n''$ and $m''$ are the number of positive and negative crossings of $D$ that belongs to $S''$, and 
$$
\operatorname{lk}(D_{S''})=\operatorname{lk}(D) - \frac{(n''-m'')}{2}. 
$$ 
Therefore,
$$
\abs{\operatorname{lk}(D)- \dfrac{1}{2}(n''-m'')}<\abs{\operatorname{lk}(D)}-\dfrac{1}{2}(n''+m'').
$$
Since $\operatorname{lk}(D)>0$, we can write 
$$
\abs{\operatorname{lk}(D)- \dfrac{1}{2}(n''-m'')}<\operatorname{lk}(D)-\dfrac{1}{2}(n''+m''),
$$  
that implies 
$$
\operatorname{lk}(D)-\dfrac{1}{2}(n''-m'')< \operatorname{lk}(D)-\dfrac{1}{2}(n''+m''), 
$$
whence 
$$
(n''-m'')>(n''+m''),
$$
 which is not possible as both $n''$ and $m''$ are not equal to zero simultaneously. 
This contradiction implies
\begin{equation}\label{EQ2} 
\abs{\operatorname{lk}(D)} - \dfrac{1}{2}\operatorname{span}(D)\leq \abs{\operatorname{lk}(D_S)} 
\end{equation}
for any $S \in \Lambda(D)$. Using eq.~(\ref{EQ1}) and eq.~(\ref{EQ2}), we can say that  $\abs{\operatorname{lk}(D)}-\operatorname{span}(D)/2= \ell_D$, where $\ell_{D} = \min \{  \operatorname{lk} (D_{S}) | S \in \Lambda (D) \}$.  

\smallskip 

\underline{\it Subcase 1.2.} Assume $\ell_{+}\leq\ell_{-}$. Then either $r_{+}< r_{-}$ or $r_{+}\geq r_{-}$.

\smallskip 

\underline{\it Subcase 1.2a.}
If $r_{+}< r_{-}$, then we have $\operatorname{lk}(D)<0$, $2\abs{\operatorname{lk}(D)}>\operatorname{span}(D)$ and $\ell_{-}> \operatorname{span}(D)$.

Since $\ell_{-}> \operatorname{span}(D)$, we can ensure that there exists a set $S\in \Lambda(D)$ such that the number of negative under linking crossings of $D$ in $S$ is $\operatorname{span}(D)$, i.e.,  $n=0$ and $m=\operatorname{span}(D)$. Therefore, 
\[\operatorname{lk}(D_S)=\operatorname{lk}(D)-\dfrac{1}{2}(n-m)=\operatorname{lk}(D)+\dfrac{1}{2}\operatorname{span}(D)<0,\]
 and
\begin{equation}\label{EQ3} \abs{\operatorname{lk}(D_S)}=-(\operatorname{lk}(D)+\dfrac{1}{2}\operatorname{span}(D))=\abs{\operatorname{lk}(D)}-\dfrac{1}{2}\operatorname{span}(D).
\end{equation}

Suppose there exist $S'\in \Lambda(D)$ such that $\abs{\operatorname{lk}(D_{S'})}=0$, then $n'+m'=\operatorname{span}(D)$ and 
\[\operatorname{lk}(D_{S'})=\operatorname{lk}(D)-\dfrac{1}{2}(n'-m')=0,\]
where $n'$ and $m'$ are the number of positive and negative crossings of $D$ that belongs to $S'$. Therefore, 
$$
\operatorname{lk}(D)=\dfrac{1}{2}(n'-m'), 
$$ 
hence
$$ 
-2\operatorname{lk}(D)=m'-n'\leq m'+n' 
$$
that implies 
$$
2\abs{\operatorname{lk}(D)}\leq \operatorname{span}(D),
$$ 
that gives a contradiction. Therefore $\abs{\operatorname{lk}(D_{S})}>0$ for every set $S\in \Lambda(D).$ 
 
Now we need to show that $\abs{\operatorname{lk}(D)}-\operatorname{span}(D)/2 \leq \ell_D$. For this let us assume that there exist a $S''\in \Lambda(D)$ such that  $\abs{\operatorname{lk}(D_{S''})}<\abs{\operatorname{lk}(D)}-\operatorname{span}(D)/2$. Then $n''+m''=\operatorname{span}(D)$, where $n''$ and $m''$ are the number of positive and negative crossings of $D$ that belongs to $S''$, and 
$$
\operatorname{lk}(D_{S''}) = \operatorname{lk} (D) - \frac{(n''-m'')}{2}.
$$ 
Therefore,
\[ \abs{\operatorname{lk}(D)- \dfrac{1}{2}(n''-m'')}<\abs{\operatorname{lk}(D)}-\dfrac{1}{2}(n''+m'').\]
Since $\operatorname{lk}(D)<0$, we have 
$$
\abs{\operatorname{lk}(D)- \dfrac{1}{2}(n''-m'')}<-\operatorname{lk}(D)-\dfrac{1}{2}(n''+m''),
$$  
that implies 
$$ 
-(-\operatorname{lk}(D)-\dfrac{1}{2}(n''+m''))<\operatorname{lk}(D)-\dfrac{1}{2}(n''-m''), 
$$
whence 
$$
(n''+m'')<(m''-n''),
$$
which is not possible as both $n''$ and $m''$ are not equal to zero simultaneously. 
This implies
\begin{equation}\label{EQ4}
 \abs{\operatorname{lk}(D)}-\dfrac{1}{2}\operatorname{span}(D)\leq \abs{\operatorname{lk}(D_S)} 
 \end{equation}
for any  $S \in \Lambda(D)$. Using eq.~(\ref{EQ3}) and eq.~(\ref{EQ4}), we can say that $\abs{\operatorname{lk}(D)}-\operatorname{span}(D)/2= \ell_D.$ 
 
\smallskip 

\underline{\it Subcase 1.2.b.} 
If $r_{+} \geq r_{-}$, then we have  $\ell_{-}\geq (\operatorname{span}(D)-2\operatorname{lk}(D))/2$ and  $r_{+}\geq (\operatorname{span}(D)+2 \operatorname{lk}(D))/2$. 

Since $\ell_{-}\geq (\operatorname{span}(D)-2\operatorname{lk}(D))/2$ and  $r_{+}\geq (\operatorname{span}(D)+2 \operatorname{lk}(D))/2$, we can ensure that there exists a set $S\in \Lambda(D)$ such that  $n=(\operatorname{span}(D)+2 \operatorname{lk}(D))/2$ and $m=(\operatorname{span}(D)-2 \operatorname{lk}(D))/2.$     Then
    \[\operatorname{lk}(D_S)=\operatorname{lk}(D)-\dfrac{1}{2}(n-m)=0.\]
Observe that whenever there exists a set $S$ in $\Lambda(D)$ such that $\operatorname{lk}(D_S)=0$, then $\ell_D=0.$ 

\smallskip 

\underline{\it Case 2.} If $r_{+} + \ell_{-} < r_{-} + \ell_{+}$, i.e., $\operatorname{span}(D)=-(r_{+} + \ell_{-} - r_{-} - \ell_{+})$, then also we have the following two subcases.

\underline{\it Subcase 2.1.} Assume $r_+> r_-$. Then $2\operatorname{lk}(D)> \operatorname{span}(D)$ and $\ell_{+}> \operatorname{span}(D)$. 
 
Since $\ell_{+}> \operatorname{span}(D)$, we can ensure that there exist a set $S\in \Lambda(D)$ such that the number of positive under linking crossings of $D$ in $S$ is $\operatorname{span}(D)$, i.e.,  $n=\operatorname{span}(D)$ and $m=0$. Therefore, 
\[\operatorname{lk}(D_S)= \operatorname{lk}(D)-\dfrac{1}{2}(n-m)=\operatorname{lk}(D)-\dfrac{1}{2}\operatorname{span}(D)>0.\] 
Hence  $\abs{\operatorname{lk}(D_S)}=\abs{\operatorname{lk}(D)-\dfrac{1}{2}\operatorname{span}(D)}$, which can be written as 
\begin{equation}\label{EQ5} 
\abs{\operatorname{lk}(D_S)} = \operatorname{lk}(D) - \dfrac{1}{2}\operatorname{span}(D), 
\end{equation}
since $\operatorname{lk}(D)>0$. To be specific, for every set $S\in \Lambda(D)$ we have $\abs{\operatorname{lk}(D_S)}>0$. Otherwise, if there exist a set $S'\in \Lambda(D)$ such that $\abs{\operatorname{lk}(D_{S'})}=0$, then $n'+m'=\operatorname{span}(D)$, where   $n'$ and $m'$ are the number of positive and negative crossings of $D$ that belongs to $S'$, and
$$
\operatorname{lk}(D_{S'})=\operatorname{lk}(D)-\dfrac{1}{2}(n'-m')=0, 
$$
that implies 
$$ 
2\operatorname{lk}(D)=n'-m'\leq n'+m' , 
$$
whence
$$ 
2\operatorname{lk}(D)\leq \operatorname{span}(D),
$$
that gives a contradiction. Therefore for every set $S\in \Lambda(D)$ we have $\abs{\operatorname{lk}(D_S)}>0,$ and hence $\ell_D \neq 0.$ 

Now we need to show that $\abs{\operatorname{lk}(D)}-\operatorname{span}(D)/2 \leq \ell_D$. For this let us assume that there exist a $S''\in \Lambda(D)$ such that  $\abs{\operatorname{lk}(D_{S''})}<\abs{\operatorname{lk}(D)}-\operatorname{span}(D)/2$. Then $n''+m''=\operatorname{span}(D)$, where $n''$ and $m''$ are the number of positive and negative crossings of $D$ that belongs to $S''$, and $\operatorname{lk}(D_{S''})=\operatorname{lk}(D)-(n''-m'')/2$. Therefore,
$$
\abs{\operatorname{lk}(D)- \dfrac{1}{2}(n''-m'')}<\abs{\operatorname{lk}(D)}-\dfrac{1}{2}(n''+m'').
$$
Since $\operatorname{lk}(D)>0$,  we can rewrite as 
$$
\abs{\operatorname{lk}(D)- \dfrac{1}{2}(n''-m'')}<\operatorname{lk}(D)-\dfrac{1}{2}(n''+m''),
$$  
that implies 
$$ 
\operatorname{lk}(D)-\dfrac{1}{2}(n''-m'')< \operatorname{lk}(D)-\dfrac{1}{2}(n''+m'')
$$ 
whence $(n''-m'')>(n''+m'')$, which is not possible as both $n''$ and $m''$ are not equal to zero simultaneously. 
This implies
\begin{equation}\label{EQ6} 
\abs{\operatorname{lk}(D)}-\dfrac{1}{2}\operatorname{span}(D)\leq \abs{\operatorname{lk}(D_S)}   
\end{equation}
for any $S \in \Lambda(D)$.  Using eq.~(\ref{EQ5}) and eq.~(\ref{EQ6}), we can say that  $\abs{\operatorname{lk}(D)}-\operatorname{span}(D)/2= \ell_D.$ 

\smallskip 

\underline {\it Subcase 2.2.} Assume $r_{+}\leq r_{-}$. Then either $\ell_{+}< \ell_{-}$ or $\ell_{+}\geq \ell_{-}$.

\smallskip 

\underline{\it Subcase 2.2.a.} 
If $\ell_{+}< \ell_{-}$, then we have $\operatorname{lk}(D)<0$, $2\abs{\operatorname{lk}(D)}>\operatorname{span}(D)$ and $r_{-}> \operatorname{span}(D)$. 

Since $r_{-}> \operatorname{span}(D)$, we can ensure that there exists a set $S\in \Lambda(D)$ such that the number of negative over linking crossings of $D$ in $S$ is $\operatorname{span}(D)$, i.e.,  $n=0$ and $m=\operatorname{span}(D)$. Therefore, 
\[\operatorname{lk}(D_S)=\operatorname{lk}(D)-\dfrac{1}{2}(n-m)=\operatorname{lk}(D)+\dfrac{1}{2}\operatorname{span}(D)<0,\]
 and
\begin{equation}\label{EQ7} \abs{\operatorname{lk}(D_S)}=-(\operatorname{lk}(D)+\dfrac{1}{2}\operatorname{span}(D))=\abs{\operatorname{lk}(D)}-\frac{1}{2}\operatorname{span}(D).\end{equation}

Suppose there exists $S'\in \Lambda(D)$ such that $\abs{\operatorname{lk}(D_{S'})}=0$, then $n'+m'=\operatorname{span}(D)$ and 
\[\operatorname{lk}(D_{S'})=\operatorname{lk}(D)-\dfrac{1}{2}(n'-m')=0,\]
where $n'$ and $m'$ are the number of positive and negative crossings of $D$ that belongs to $S'$. Therefore, 
$$
\operatorname{lk}(D)=\dfrac{1}{2}(n'-m'),  
$$
whence 
$$
-2\operatorname{lk}(D)=m'-n'\leq m'+n', 
$$ 
that implies 
$$
2\abs{\operatorname{lk}(D)}\leq \operatorname{span}(D),
$$
which gives a contradiction. Therefore $\abs{\operatorname{lk}(D_{S})}>0$ for every set $S\in \Lambda(D).$ 
 
 Now we need to show that $\abs{\operatorname{lk}(D)}-\operatorname{span}(D)/2 \leq \ell_D$. For this let us assume that there exist a set $S''\in \Lambda(D)$ such that  $\abs{\operatorname{lk}(D_{S''})}<\abs{\operatorname{lk}(D)}-\operatorname{span}(D)/2$. Then $n''+m''=\operatorname{span}(D)$, where $n''$ and $m''$ are the number of positive and negative crossings of $D$ that belongs to $S''$, and $\operatorname{lk}(D_{S''})=lk(D)-(n''-m'')/2$. Therefore,
\[ \abs{\operatorname{lk}(D)- \dfrac{1}{2}(n''-m'')}<\abs{\operatorname{lk}(D)}-\dfrac{1}{2}(n''+m'').\]
Since $\operatorname{lk}(D)<0$, we can rewrite 
$$
\abs{\operatorname{lk}(D)- \dfrac{1}{2}(n''-m'')}<-\operatorname{lk}(D)-\dfrac{1}{2}(n''+m''),
$$
whence   
$$ 
-(-\operatorname{lk}(D)-\dfrac{1}{2}(n''+m''))<\operatorname{lk}(D)-\dfrac{1}{2}(n''-m''), 
$$
that implies 
$$
(n''+m'')<(m''-n''),
$$
which is not possible as both $n''$ and $m''$ are not equal to zero simultaneously. 
This implies
\begin{equation}\label{EQ8}
\abs{\operatorname{lk}(D)}-\dfrac{1}{2}\operatorname{span}(D)\leq \abs{\operatorname{lk}(D_S)}  
\end{equation}
for any $S \in \Lambda(D)$. Using eq.~(\ref{EQ7}) and eq.~(\ref{EQ8}), we can say that $\abs{\operatorname{lk}(D)}-\operatorname{span}(D)/2= \ell_D.$ 
 
 \smallskip 

 \underline{\it Subcase 2.2b.}
 If $\ell_{+} \geq \ell_{-}$, then we have  $r_{-}\geq (\operatorname{span}(D)-2\operatorname{lk}(D))/2$ and  $\ell_{+}\geq (\operatorname{span}(D)+2 \operatorname{lk}(D))/2$. 

Since $r_{-}\geq (\operatorname{span}(D)-2\operatorname{lk}(D))/2$ and  $\ell_{+}\geq (\operatorname{span}(D)+2 \operatorname{lk}(D))/2$, we can ensure that there exists a set  $S\in \Lambda(D)$ such that  $n=(\operatorname{span}(D)+2 \operatorname{lk}(D))/2$ and $m=(\operatorname{span}(D)-2 \operatorname{lk}(D))/2$, then
    \[\operatorname{lk}(D_S)=\operatorname{lk}(D)-\dfrac{1}{2}(n-m)=0.\]
Observe that whenever there exist a set $S$ in $\Lambda(D)$ such that $\operatorname{lk}(D_S)=0$, then $\ell_D=0.$


Thus, in all considered cases $\ell_D=0$ or $\ell_D=\abs{\operatorname{lk}(D)}-\operatorname{span}(D)/2$ and theorem is proved. 
\end{proof}

\begin{fact} \label{fact1} 
{\rm By changing one crossing in $D=D_{1} \cup D_{2}$, the value $\ell_D$ either changes by $\pm 1$ or remains same.}
\end{fact}

\begin{proposition}\label{p2}
The following properties hold: \\ 
(1) If $D=D_{1}\cup D_{2}$ is virtual link diagram with all linking crossings of same sign, then
$$
\ell_{D} = \abs{ \operatorname{lk}(D)} - \frac{1}{2} \operatorname{span}(D).
$$
(2) $\ell_{D}$ is an invariant for 2-component virtual links $L$. \\ 
(3) For a n-component virtual link $L = L_{1} \cup L_{2} \cup \ldots \cup L_{n}$ the value $\sum_{i \neq j} \ell_{L_{i} \cup L_{j}}$ is an invariant.
\end{proposition}

\begin{proof}
(1) Let $D=D_{1}\cup D_{2}$ be a virtual link diagram with all linking crossings of same sign. Then $2 \abs{\operatorname{lk}(D)} \geq \operatorname{span}(D)$ and using Theorem~\ref{thm-span}, we have
$$
\ell_{D}=\abs{ \operatorname{lk}(D)} - \frac{1}{2} \operatorname{span}(D).
$$

(2) Let $D$ and $D'$ be two diagrams of $L=K_1\cup K_2$. Since linking number and span are invariants, either
$2\abs{\operatorname{lk}(D)} \geq \operatorname{span}(D)$ and $2\abs{\operatorname{lk}(D')} \geq \operatorname{span}(D')$ or $2\abs{\operatorname{lk}(D)} < \operatorname{span}(D)$ and $2\abs{\operatorname{lk}(D')} < \operatorname{span}(D')$.

Using the proof of Theorem~\ref{thm-span}, we have either $\ell_{D} = \ell_{D'} = 0$ or $\ell_{D} = \abs{ \operatorname{lk}(D)}-\frac{1}{2} \operatorname{span}(D)$ and $\ell_{D'}=\abs{ \operatorname{lk}(D')} - \frac{1}{2} \operatorname{span}(D')$. Hence $\ell_D = \ell_{D'}$.

(3) Since each $\ell_{L_{i}\cup L_{j}}$, $i\neq j$, is a positive quantity and an invariant, $\sum_{i\neq j} \ell_{L_{i}\cup L_{j}}$ is an invariant for~$L$.
\end{proof}

\begin{theorem} \label{thm-bound}
If $D=D_{1}\cup D_{2}\cup \ldots \cup D_{n}$ is a diagram of a virtual link $L$, then
$$
\Big(\operatorname{span}(D), \sum_{i \neq j} \ell_{D_{i} \cup D_{j}} + \frac{1}{2} \sum_{i=1}^{n} \sum_{k \in \mathbb{Z}\setminus \{0\}} |  J_k(D_{i}) | \Big) \leq U(L).
$$ 
\end{theorem}

\begin{proof} If  $U(L)=(m,n)$, then by using Corollary~\ref{c1}, we can see that $m \geq \operatorname{span}(D)$. 

Now consider a set $S\in  \Lambda(D)$ such that $\operatorname{lk}(D'_i \cup D'_j) = \ell_{D'_i \cup D'_j}$, where $D_S=D'_{1}\cup D'_{2} \cup \ldots \cup D'_{n}$ and $1\leq i, j \leq n$. A necessary condition for a virtual link $L=K_{1}\cup K_{2}\cup \ldots \cup K_{n}$ to be unlink is,  $\operatorname{lk}(K_{i}\cup K_{j})=0$ for each $i \neq j$ and each $K_i$ is trivial. Using the Fact~\ref{fact1}, Proposition~\ref{p2}(1) and Proposition~\ref{p1}, we have
$$ 
n \geq \sum_{i\neq j} \ell_{D_{i}\cup D_{j}} + \frac{1}{2} \sum_{i=1}^{n} \sum_{k \in \mathbb{Z} \setminus \{0\}} | J_k(D_{i}) | . 
$$
This completes the proof. 
\end{proof}

\begin{theorem} \label{thm-upb}
If $D$ is a diagram of a virtual link $L$ with $v$ number of virtual crossings, then 
$$U(L) \leq (v,d(D)).$$ 
\end{theorem}

\begin{proof} 
Consider a virtual link diagram $D$ with $v$ number of virtual crossings. Let $D'$ be the diagram obtained from $D$ by inserting $v$ number of RII moves locally as shown in Fig.~\ref{upper bound on u(L)}(b). 
\begin{figure}[!ht]
\begin{center}
\includegraphics[scale=.5]{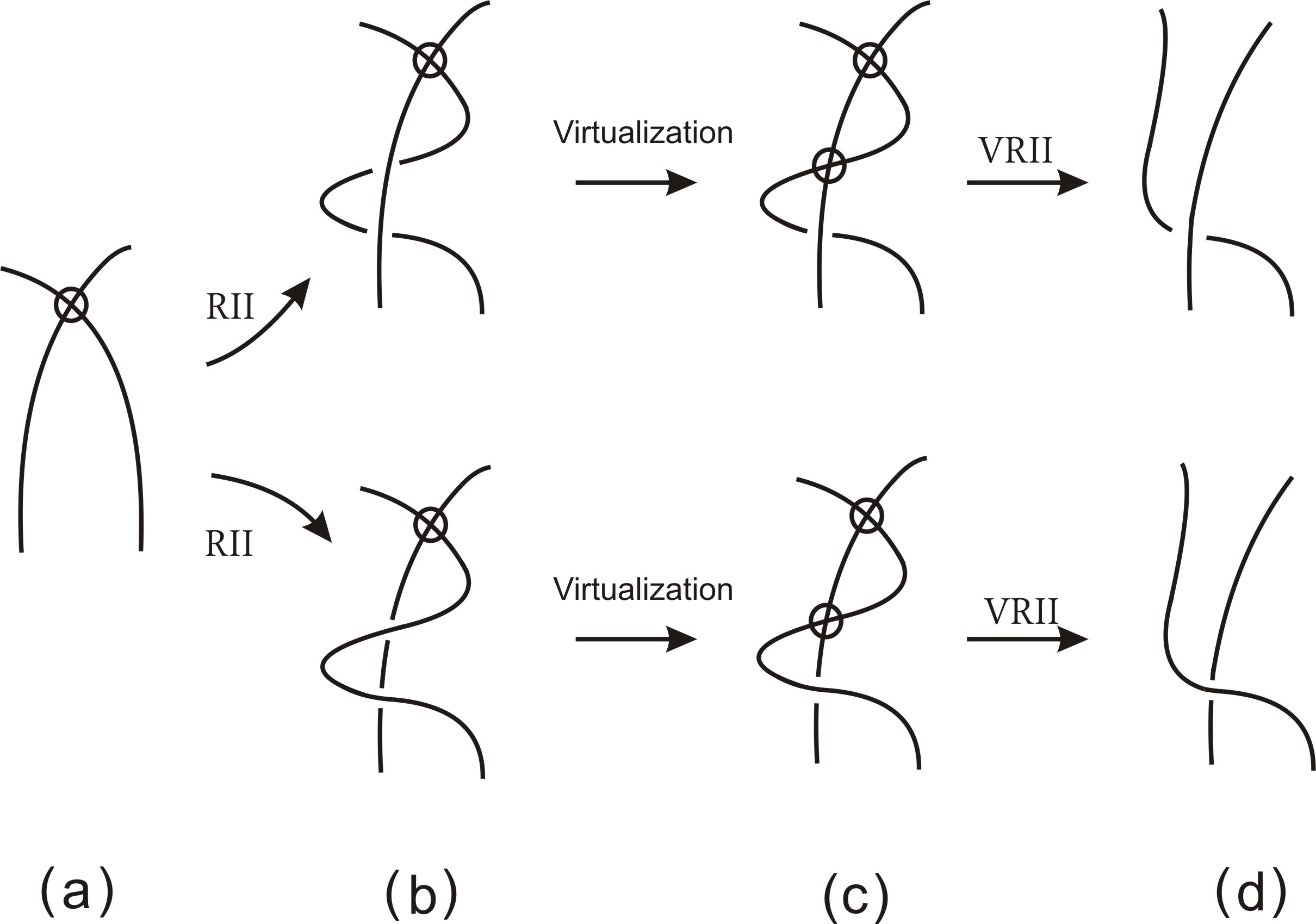}
\caption{Diagram transformations.}
\label{upper bound on u(L)}
\end{center}
\end{figure}
Obtain a diagram $D''$ from $D'$ by virtualizing $v$ number of crossings and applying VRII moves as illustrated in Fig.~\ref{upper bound on u(L)}(c) and Fig.~\ref{upper bound on u(L)}(d). Since there are two kinds of RII moves, one can obtain $2^{v}$ classical link diagrams from $D$ by replacing $v$ virtual crossings to classical crossings.  Out of these $2^{v}$  diagrams, there exists at least one diagram (say $D_{1}$) for which $d(D_{1})= d(D)$. We can ensure existence of such a diagram from the fact that: corresponding to any based point $a$ in $D$ we can resolve the virtual crossings to non-warping crossing points with respect to based point $a$. Further as $D_{1}$ is a classical diagram, we have $U(D_{1}) \leq (0,d(D))$. Hence $U(D)\leq (v,d(D))$ and $U(L)\leq (v,d(D))$. 
\end{proof}

\begin{corollary}
If $D$ is a diagram of one-component virtual link $L$ with $v$ virtual crossings and $n$ classical crossings, then 
$$
U(L) \leq \Big(v-1,\Big \lfloor \dfrac{n}{2}\Big \rfloor \Big).
$$
\end{corollary}

\begin{proof} 
Let $D'$ be a diagram obtained from $D$ by resolving $v-1$ number of virtual crossings to classical crossings, as discussed in Theorem~\ref{thm-upb}. Choose a based point $p$ in the neighborhood of virtual crossing point of $D'$. By changing all the warping crossing points from under to over, while moving from $p$ along the orientation, the resulting diagram presents a diagram of trivial knot. Hence  
$
U(L)\leq U(D)\leq (v-1,d(D_{p})) \leq (v-1,\Big \lfloor \dfrac{n}{2}\Big \rfloor).
$ 
\end{proof}

Theorem~\ref{thm-upb} implies the following observation.   

\begin{corollary} \label{c3}
 If $D$ is a diagram of virtual link $L$ having $v$ virtual crossings, which is minimum over all the diagrams of $L$, then $
U(L)\leq U(D)\leq (v, d(D))$. 
\end{corollary}

\section{Unknotting index for virtual pretzel links} \label{section3}

In this section we will provide unknotting index for a large class of virtual links obtained from pretzel links by virtualizing some of its classical crossings. For examples of computing unknotting numbers of certain virtual torus knots see~\cite{ishikawa}.

Let $D$ be the standard diagram of a pretzel link $L(p_{1},p_{2}, \ldots, p_{n})$ with labelling as illustrated in Fig.~\ref{fig8}.   For a reader convenience, whenever we virtualized some crossings in $L(p_{1},p_{2}, \ldots, p_{n})$, then labelling in the resulting diagram remains same. Also, we call the crossings at positions 
$$
\sum_{k=0}^{i-1} p_{k}+1, \quad \sum_{k=0}^{i-1} p_{k}+2, \quad \ldots, \quad \sum_{k=0}^{i-1} p_{k}+p_{i}
$$ 
as $i$--th strand crossings, where $p_{0}=0$.
 \begin{figure}[!ht]
{\centering
\includegraphics[scale=0.5]{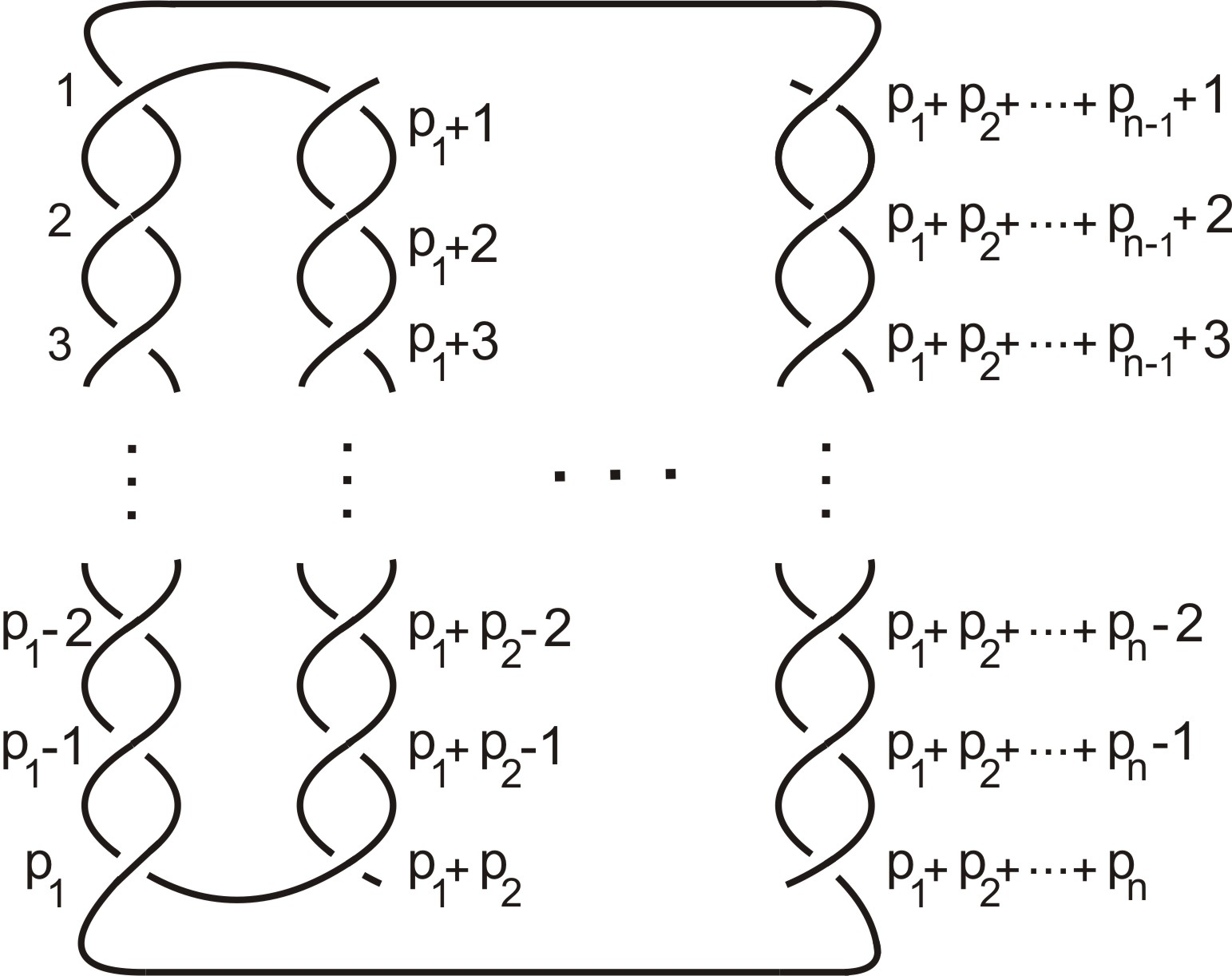}
\caption{Standard diagram of pretzel link $L(p_{1},p_{2},\ldots,p_{n})$.}
\label{fig8}
}
\end{figure}

\begin{theorem} \label{thm1-pretzel}
Let $D$ be a diagram of a virtual link $L$ obtained from labelled pretzel link $L'(p_{1},p_{2} \ldots, p_{n})$ by virtualizing $k$ crossings with $n$ even. If $k_{1}$ is the number of virtual crossings among $k$ having even labelling and each $p_{i}$ is odd, then
$$
U(L) = \begin{cases}
( | k-2k_{1} |, \displaystyle \frac{1}{2}  \sum_{i=1}^{n} p_{i}-k+k_{1} ), & \text{if}~~k>2k_{1}, \\
( | k-2k_{1} |, \displaystyle \frac{1}{2}  \sum_{i=1}^{n} p_{i}-k_{1} ),    & otherwise.
\end{cases}
$$
\end{theorem}

\begin{proof} 
Since each $p_i$ is odd and $n$ is even, $L'(p_{1},p_{2} \ldots, p_{n})$ represents a $2$-component link. Let $C_{1}$ and $C_{2}$ be the set of even and odd labellings, respectively. Let $D$ be the virtual link obtained from $L'(p_{1}, p_{2}, \ldots, p_{n})$ by virtualizing $k_{1}$ and $k_{2}$ crossings from the labelling set $C_{1}$ and $C_{2}$, respectively. If $D$ represents virtual link $L$ and $k=k_{1}+k_{2}$, then $\operatorname{span}(L)=\mid k-2k_{1}\mid$. Since all the linking crossings in $D$ are of same sign and $\abs{\operatorname{lk}(D)}= \frac{1}{2}(\sum p_i-k)$, by using Theorem~\ref{thm-bound} and Proposition~\ref{p2}(a), 
$$
U(L)\geq \Big(\abs{ k-2k_{1}},\frac{1}{2} \sum _{i=1}^{n}p_{i}-\frac{1}{2}(k+\abs{ k-2k_{1}})\Big).
$$ 
Since $\operatorname{span}(D)=\abs{k-2k_{1}}$, there exist a diagram $D'$ obtained from $D$ by virtualizing $\abs{ k-2k_{1}}$ crossings such that $\operatorname{span}(D')=0.$ Now diagram $D'$ has $\sum _{i=1}^{n}p_{i}-k-\abs{ k-2k_{1}}$  crossings, all of same sign and $\operatorname{span}(D')=0$. Therefore, number of under linking crossings and number of over linking crossings in $D'$ are equal. More precisely, number of crossings in $D'$ with the same parity of labelling is $\frac{1}{2} (\sum _{i=1}^{n}p_{i}-k-\abs{ k-2k_{1}})$.

By applying crossing change operation to all the crossings in $D'$ with odd labelling, the resulting diagram becomes trivial. Hence 
$$
U(L)\leq \Big(\abs{ k-2k_{1}}, \frac{1}{2} \sum_{i=1}^{n}p_{i}-\frac{1}{2}(k+\abs{k-2k_{1}})\Big).
$$ 
It is easy to observe through Gauss diagram corresponding to $D'$. 
\end{proof}

\begin{example} {\rm 
Let $L'$ be a virtual link diagram obtained from $L(7,5,9,11)$ pretzel link by virtualizing $13$ crossings as shown in Fig.~\ref{fig9a}. For this specific example,  $k=13$ and  $k_{1}=10$. Thus $\operatorname{span}(D)=7$ and using Theorem~\ref{thm1-pretzel}, we have $U(L)\geq (7,6)$. 
 \begin{figure}[!ht]
{\centering
\subfigure[$ L'= K_1 \cup K_2$]{ \includegraphics[scale=0.5]{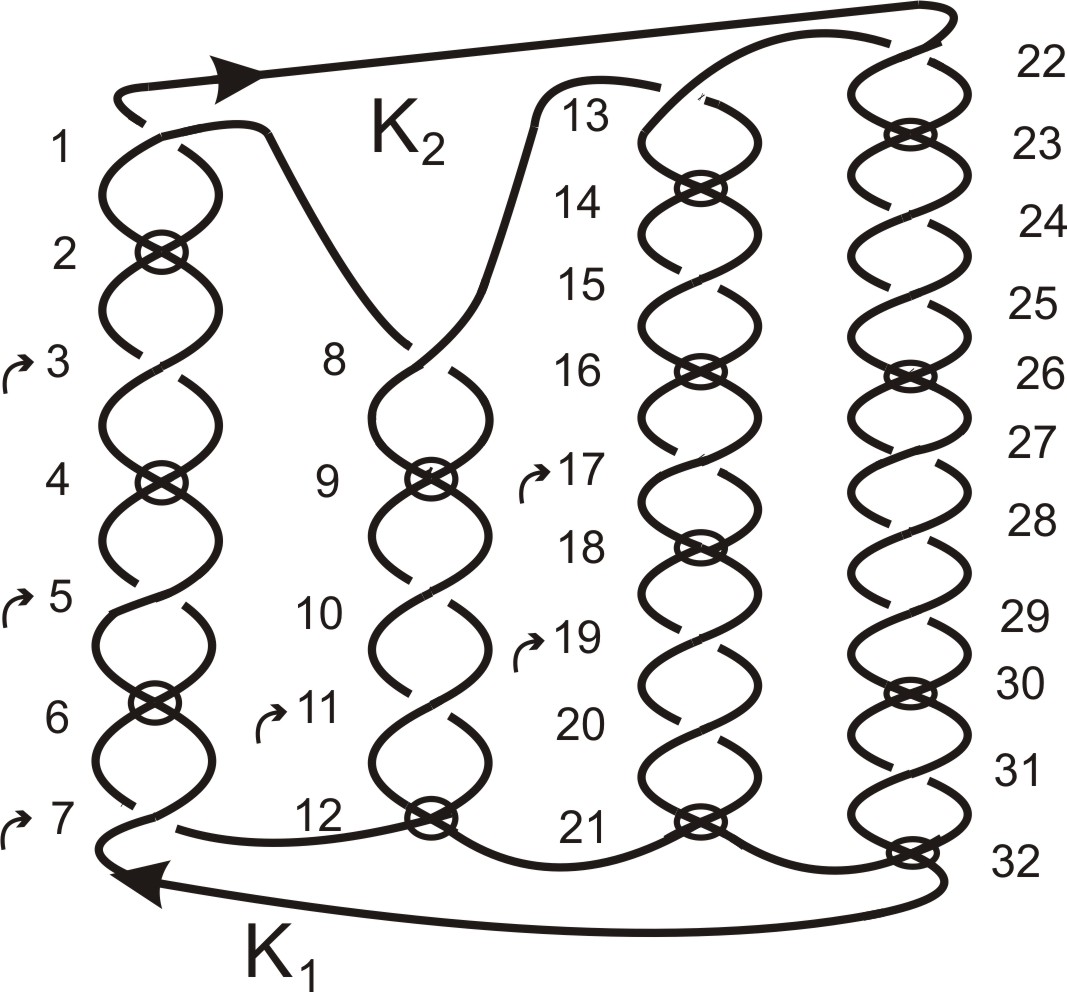} \label{fig9a} }
\hspace{.5cm}
\subfigure[$G=G(K_1)\cup G(K_2)$] { \includegraphics[scale=0.5]{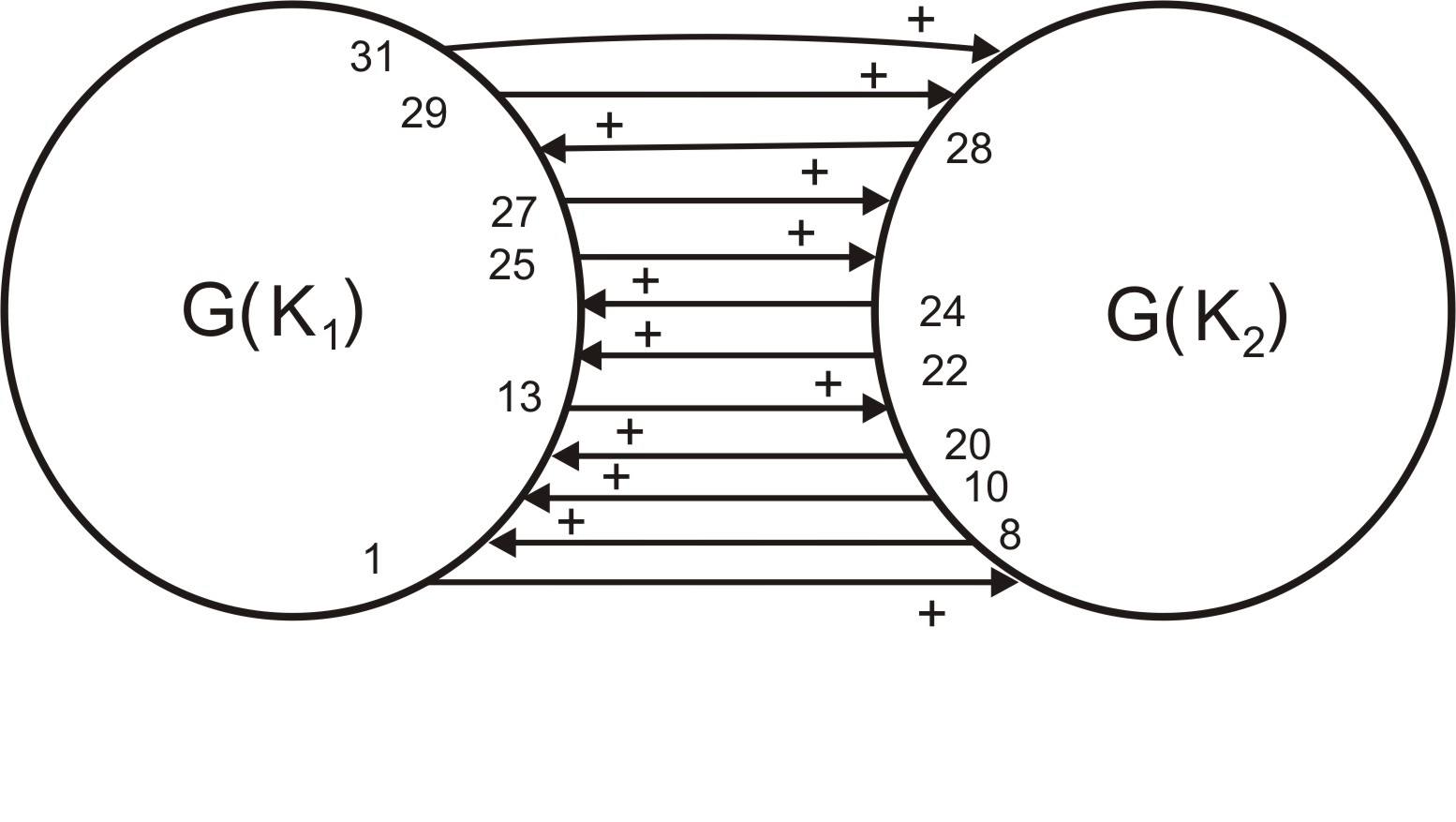} \label{fig9b} }
\caption{Virtualization of a pretzel link diagram and Gauss diagram.}
\label{31}
}
\end{figure}
By virtualizing crossings labelled with $3$, $5$, $7$, $11$, $17$ and $19$ integers in $D$, the resulting diagram, say $D'$, has zero span. Let $G$ be the Gauss diagram corresponding to $D'$ as shown in Fig.~\ref{fig9b}. Now if we apply crossing change operations on the crossings  $1$, $13$, $25$, $27$, $29$ and $31$ in $D'$, $G$ becomes a Gauss diagram of trivial link. Hence $U(L) \leq(7,6)$.
}
\end{example}

\begin{corollary} \label{U(L)=P(p1,p2)}
Let $D$ be a diagram of virtual link $L$ obtained from labelled pretzel link $L'(p_{1},p_{2})$ by virtualizing $k$ crossings. If $p_{1}+ p_{2}$ is even, then
$$
U(L)=
 \begin{cases}
 \left( |k-2k_{1}|, \displaystyle  \frac{p_{1}+p_{2}-2k+2k_{1}}{2} \right),             & \text{if}~~~k>2k_{1}, \\
 \left( |k-2k_{1}|, \displaystyle \frac{p_{1}+p_{2}-2k_{1}}{2} \right) ,              & otherwise,
\end{cases}
$$
where $k_{1}$  is the number of crossings virtualized that are labelled with even integers in $L'$.
\end{corollary}

\begin{proof} If $p_{1}+ p_{2}$ is even, then $L'(p_{1},p_{2})$ represents a diagram of $(2,p)$ torus link. Therefore, there exist two positive odd integers $p'_{1} ,p'_{2}$ and a diagram $D'$ of $L$, which is obtained from $L'(p'_{1},p'_{2})$ by virtualizing $k$   crossings with the same parity of labelling. Now proof follows directly from Theorem \ref{thm1-pretzel}. 
\end{proof}

\begin{theorem}\label{thm2-pretzel}
Let $D$ be a diagram of virtual link $L$ obtained by virtualizing $k$ classical crossings from the pretzel link $L'(p_{1},p_{2},\ldots,p_{n})$. For an even $n$, if all $p_{1}, p_{2}, \ldots, p_{n}$ are also even, then
$$
U(L)=\Big( \sum_{i=1}^{n}\mid k'_{i}-k_{i} \mid, \sum_{i=1}^{n}\frac{p_{i}-(k_{i}+k'_{i})-\mid k'_{i}-k_{i} \mid}{2}\Big),
$$
where $k'_{i}$ (respectively, $k_{i})$ is the number of crossings virtualized, that are labelled with odd (respectively,  even) integers in $i$--th strand.
\end{theorem}

\begin{proof} Let  $E_{i}$ and $O_{i}$ denotes the set of crossings labelled with even and odd integers in the $i$--th strand, respectively. Suppose $D$ is a virtual link diagram obtained from $L'(p_{1},p_{2},\ldots,p_{n})$ by virtualizing $k$ crossings at even labelling and $k'$ crossings at odd labelling, respectively. Then $k= \sum_{i=1}^{n}k_{i}$ and  $k'= \sum_{i=1}^{n}k'_{i}$, where $k_{i}$ is the number of crossings from the labelling set $E_{i}$ and $k'_{i}$ is the number of crossings from the labelling set $O_{i}$, respectively.

Because each $p_{i}$ and $n$ are even, $D$ is a $n$-component link. We can represent $D=D_{1}\cup D_{2}\cup \ldots \cup D_{n}$, such that crossings of $D_{1}\cup D_{n}$ and $ D_{i-1}\cup D_{i} $, $2\leq i \leq n$ are represented by crossings of $1$--st strand and $i$--th strand, respectively. Since all the linking crossings of $D_{i}\cup D_{j}$, where $i\neq j$ and $1\leq i,j\leq n$, are of the same sign and there is no crossing between $D_{i}$ and $ D_{j}$ with $2\leq \mid i-j\mid \neq n-1$,  span of $D$ is given by
$$
\begin{gathered}
\operatorname{span} (D) = \operatorname{span} (D_{1} \cup D_{n}) + \sum _{i=2}^{n} \operatorname{span} (D_{i-1}\cup D_{i}) 
\\ = \mid k_{1}-k'_{1}\mid+ \sum _{i=2}^{n}\mid k_{i}-k'_{i}\mid= \sum _{i=1}^{n}\mid k_{i}-k'_{i}\mid.
\end{gathered}
$$   
From Proposition~\ref{p2}(a) and Theorem~\ref{thm-bound}, we have $\ell_D=\frac{1}{2} \sum _{i=1}^{n}(p_{i}-(k_{i}+k'_{i})-\mid k_{i}-k'_{i}\mid )$ and
$$
U(L)\geq \Big( \sum _{i=1}^{n}\mid k_{i}-k'_{i}\mid,\frac{1}{2} \sum _{i=1}^{n}(p_{i}-(k_{i}+k'_{i})-\mid k_{i}-k'_{i}\mid )\Big).
$$
Since $\operatorname{span}(D)= \sum _{i\neq j} |k_{i}-k'_{i}|$, there exists a diagram $D'$ obtained from $D$ by virtualizing $\sum_{i\neq j} |k_{i}-k'_{i}|$ crossings such that $\operatorname{span}(D')=0$.

It is obvious that number of crossing in $D'_{1}\cup D'_{n}$ and $ D'_{i-1}\cup D'_{i}$, $2\leq i \leq n$, are $(p_{1}-k_{i}-k'_{1}-\mid k_{1}-k'_{1} \mid)$  and $p_{i}-(k_{i}+k'_{i})-\mid k_{i}-k'_{i} \mid$, respectively.

Now with the similar argument given in Theorem~ \ref{thm1-pretzel}, all the linking crossings of $D'_{1}\cup D'_{n}$ and $ D'_{i-1}\cup D'_{i}$, $2\leq i \leq n$, can be removed by applying crossing change operation on $\frac{1}{2}(p_{1}-k_{1}-k'_{1}-\mid k_{1}-k'_{1} \mid)$ and $\frac{1}{2}(p_{i}-(k_{i}+k'_{i})-\mid k_{i}-k'_{i} \mid)$ number of crossings, respectively. Since $D'$ has no self-crossing and all the linking crossings in $D'$ are removed by changing $\displaystyle \sum_{i=1}^{n}\frac{p_{i}-(k_{i}+k'_{i})-\mid k_{i}-k'_{i} \mid}{2} $ crossings, hence
$$ 
U(L) \leq \Big ( \sum_{i=1}^{n}\mid k_{i}-k'_{i} \mid, \sum_{i=1}^{n}\frac{p_{i}-(k_{i}+k'_{i})-\mid k_{i}-k'_{i} \mid }{2} \Big).
$$
The proof is completed. 
\end{proof}


\section*{Acknowledgements}

The first and second named authors were supported by DST -- RSF Project INT/RUS/RSF/P-2. The third named author was supported by the Russian Science Foundation (grant no. 16-41-02006). K.~K. and M.~P. would like to thanks to Prof. Akio Kawauchi and Prof. Seiichi Kamada for their valuable discussions and suggestions. 


\end{document}